%% file: paper.tex
\documentclass[10pt, oneside, a4paper, abstracton]{scrartcl}

\usepackage[american]{babel}

\usepackage[noadjust]{cite}

\usepackage{amsmath}
\usepackage{amssymb}
\usepackage{amsthm}

\usepackage{tikz}
\usepackage{pgfplots}
\pgfplotsset{compat = newest}

\usepackage{algpseudocode}
\usepackage{algorithm}

\definecolor{myRed}{HTML}{E34A33}
\definecolor{myBlue}{HTML}{0571B0}
\definecolor{myBrown}{HTML}{A6611A}
\definecolor{linkBlue}{HTML}{0055C9}
\definecolor{linkRed}{HTML}{FF1A24}
\definecolor{linkPurple}{HTML}{6200D9}

\usepackage[%
  colorlinks = true,
  linkcolor  = linkBlue,
  citecolor  = linkRed,
  urlcolor   = linkPurple
]{hyperref}
\usepackage{url}

\usepackage{geometry}
\theoremstyle{definition}\newtheorem{definition}{Definition}
\theoremstyle{plain}\newtheorem{theorem}{Theorem}
\theoremstyle{plain}\newtheorem{lemma}{Lemma}

\begin{document}

\title{Hankel-Norm Approximation of Large-Scale Descriptor Systems}

\author{%
  Peter~Benner\thanks{
   Max Planck Institute for Dynamics of Complex Technical Systems,
   Sandtorstr. 1, 39106 Magdeburg, Germany.\newline
   E-mail: \texttt{\href{mailto:benner@mpi-magdeburg.mpg.de}%
     {benner@mpi-magdeburg.mpg.de}}
   \newline
   Faculty of Mathematics, Otto von Guericke University,
   Universit{\"a}tsplatz 2, 39106 Magdeburg, Germany.\newline
   E-mail: \texttt{\href{mailto:peter.benner@ovgu.de}%
     {peter.benner@ovgu.de}}},
  Steffen~W.~R.~Werner\thanks{
    Max Planck Institute for Dynamics of Complex
    Technical Systems, Sandtorstr. 1, 39106 Magdeburg, Germany.\newline
    E-mail: \texttt{\href{mailto:werner@mpi-magdeburg.mpg.de}%
      {werner@mpi-magdeburg.mpg.de}}}
}

\date{~}

\maketitle

\begin{abstract}
  \input{sections/abstract.tex}

  \vspace{1em}
\end{abstract}

\input{sections/introduction.tex}

\input{sections/basics.tex}

\input{sections/ghna.tex}

\input{sections/methods.tex}

\input{sections/examples.tex}

\input{sections/conclusion.tex}

\section*{Acknowledgment}
\input{sections/acknowledgment.tex}

\bibliographystyle{plain}
\bibliography{bibtex/myrefs}
\addcontentsline{toc}{section}{References}

\begin{appendix}
  \input{sections/appendix.tex}
\end{appendix}

\end{document}

%% file: sections/abstract.tex
The Hankel-norm approximation is a model reduction method which provides the 
best approximation in the Hankel semi-norm.
In this paper the computation of the optimal Hankel-norm approximation is 
generalized to the case of linear time-invariant continuous-time descriptor 
systems.
An efficient algorithm is developed by refining the generalized balanced 
truncation square root method.
For a wide practical usage, adaptations of the introduced algorithm towards 
stable computations and sparse systems are made as well as an approach for a 
projection-free algorithm.
To show the approximation behavior of the introduced method, numerical 
examples are presented.

%% file: sections/introduction.tex
\section{Introduction}

Many different real-world applications, like chemical 
processes, electrical circuits and networks, or computational fluid dynamics, 
naturally lead to models, described by systems of differential-algebraic
equations.
Since experiments can be very costly, time-consuming, and expensive, these 
models are used for simulations and the design of controllers.
The modeling process often results in linear time-invariant continuous-time
descriptor systems of the form
\begin{align} \label{eq:desc}
  \begin{aligned}
    E\dot{x}(t) & = Ax(t) + Bu(t), \\
    y(t) & = Cx(t) + Du(t),
  \end{aligned}
\end{align}
with $E, A \in \mathbb{R}^{n \times n}$, $B \in \mathbb{R}^{n \times m}$, 
$C \in \mathbb{R}^{p \times n}$, $D \in \mathbb{R}^{p \times m}$.
Here, $u(t) \in \mathbb{R}^{m}$ are the inputs of the system, which influence
the generalized states $x(t) \in \mathbb{R}^{n}$ to get the desired outputs
$y(t) \in \mathbb{R}^{p}$.
Throughout this paper, it is assumed that the matrix pencil
$\lambda E - A$ is regular, i.e., there exists at least one
$\lambda \in \mathbb{C}$ such that $\mathrm{det}(\lambda E - A) \neq 0$.
In this case, and with the initial condition $Ex(0) = 0$, the input-output
behavior of the system \eqref{eq:desc} in the frequency domain can be described
via the system's transfer function 
\begin{align} \label{eq:transfun}
  G(s) & = C(sE - A)^{-1}B + D.
\end{align}
The quintuple $(E, A, B, C, D)$, consisting of matrices from \eqref{eq:desc},
defines a realization of \eqref{eq:desc} and its transfer function
\eqref{eq:transfun}.
Usually, the numbers of inputs and outputs are very small in contrast to the 
number of differential-algebraic equations and generalized states $n$, which
quickly enlarges due to different reasons, e.g., the model shall provide a
required accuracy.
Because of that, the usage of complete models often reaches the limits of 
computational resources like memory and computation time.
Since the acquired data for the model usually contain a huge amount of 
redundancies, it is possible to approximate the original model by a new system 
with a much smaller order.
The task of model reduction is to construct a reduced-order descriptor system
\begin{align} \label{eq:reddesc}
  \begin{aligned}
    \hat{E}\dot{\hat{x}}(t) & = \hat{A}\hat{x}(t) + \hat{B}u(t), \\
    \hat{y}(t) & = \hat{C}\hat{x}(t) + \hat{D}u(t),
  \end{aligned}
\end{align}
of order $r \ll n$, such that the input-output behavior of the original system 
\eqref{eq:desc} is approximated. 

Many model reduction techniques were originally developed for the standard 
system case, where the descriptor term $E$ is the identity matrix $I_ {n}$ 
(or at least nonsingular).
But in recent years, quite a few of these methods have been extended to the 
case of descriptor systems with singular $E$ matrices.
There are different approaches for the construction of \eqref{eq:reddesc}, e.g., 
matrix equations can be used to determine a measure for truncatable states 
\cite{morBenS17}, or the transfer function can be approximated by rational 
interpolation \cite{morGugSW13}.
A special technique of model reduction is the computation of the optimal 
Hankel-norm approximation (HNA).
This technique actually provides a best approximation in the Hankel semi-norm.
Based on the work of Adamjan, Arov, and Krein about the approximation of Hankel
matrices \cite{AdaAK71}, an algorithm for the computation of the HNA for
standard systems was introduced by Glover in \cite{morGlo84}.

A generalization of the HNA to the descriptor system case was already 
mentioned by Cao, Saltik, and Weiland in \cite{morCaoSW15}.
They are using the Weierstrass canonical form for an explicit construction of 
reduced decoupled subsystems.
The main problem of this method is the computation of the Weierstrass canonical 
form which is numerically costly and unstable.
Also additional conditions, like C-controllability and C-observability of the 
system, have to be assumed.

In this paper, a new efficient algorithm for the computation of the 
generalized Hankel-norm approximation (GHNA) will be proposed. Our main
contributions are twofold:
\begin{enumerate}
  \item We generalize the concept of all-pass transfer functions to descriptor
    systems (Theorem \ref{thm:allpass}).
  \item We derive new and reliable numerical implementations of the GHNA that
    also allow the application of the Hankel-norm approximation method to
    large-scale problems with sparse coefficient matrices as they arise, e.g.,
    from systems with dynamics described by semi-discretized unsteady partial
    differential equations.
\end{enumerate}
Therefor, in Section \ref{sec:basics} the mathematical background of linear 
descriptor systems is recalled.
Then, the HNA method for the standard system case is introduced in the first part
of Section \ref{sec:GHNA}.
Afterwards, the generalized balanced truncation is reviewed and used for the
construction of the new GHNA method.
The numerical difficulties and adjustments are discussed in Section 
\ref{sec:methods} for usable implementations of the method.
Two different implementations of the method are then tested on numerical 
examples in Section \ref{sec:example}.
In Section \ref{sec:conclusion}, the conclusions of this paper can be found.

%% file: sections/basics.tex
\section{Mathematical Basics}
\label{sec:basics}

For regular matrix pencils $\lambda E - A$, the Weierstrass canonical form
always exists: there are invertible matrices $W, T \in
\mathbb{C}^{n \times n}$ such that
\begin{align} \label{eq:weierform}
  W(\lambda E - A)T & = \lambda \begin{bmatrix} I_{n_{f}} & 0 \\ 0 & N 
    \end{bmatrix} - \begin{bmatrix} J & 0 \\ 0 & I_{n_{\infty}} \end{bmatrix},
\end{align}
where $J$ and $N$ are both in Jordan canonical form, $J$ is regular, and $N$ is 
nilpotent with index $\nu$.
The numbers $n_{f}$ and $n_{\infty}$ are the dimensions of the deflating 
subspaces corresponding to the finite and infinite eigenvalues of 
$\lambda E - A$, respectively.
Then, the spectral projectors onto the left and right deflating subspaces 
corresponding to the finite eigenvalues of the matrix pencil $\lambda E - A$ 
can be defined as
\begin{align} \label{eq:projectors}
  \begin{aligned}
    P_{\ell} & = W \begin{bmatrix} I_{n_{f}} & 0 \\ 0 & 0 \end{bmatrix} W^{-1} 
      & \text{and} && P_{r} & = T^{-1} \begin{bmatrix} I_{n_{f}} & 0 \\ 0 & 0 
    \end{bmatrix} T,
  \end{aligned}
\end{align}
with $W$ and $T$ from the Weierstrass canonical form \eqref{eq:weierform}.

Another necessary assumption is the c-stability of the matrix pencil 
$\lambda E - A$, i.e., the matrix pencil $\lambda E - A$ is regular and all
finite eigenvalues of $\lambda E - A$ lie in the open left half-plane.
In this case, the proper controllability and observability Gramians are defined 
as the unique, positive semidefinite solutions of the projected generalized 
continuous-time Lyapunov equations
\begin{align} \label{eq:gcalec}
  E\mathcal{G}_{pc}A^{T} + A\mathcal{G}_{pc}E^{T} + P_{\ell}BB^{T}P_{\ell}^{T} 
    & = 0, ~~~~\mathcal{G}_{pc} = P_{r}\mathcal{G}_{pc}P_{r}^{T},\\ 
  \label{eq:gcaleo}
  E^{T}\mathcal{G}_{po}A + A^{T}\mathcal{G}_{po}E + P_{r}^{T}C^{T}CP_{r} 
    & = 0, ~~~~\mathcal{G}_{po} = P_{\ell}^{T}\mathcal{G}_{po}P_{\ell},
\end{align}
with $P_{\ell}$ and $P_{r}$ the spectral projectors corresponding to the finite 
eigenvalues \eqref{eq:projectors}; see \cite{Sty02}.
Furthermore, the improper controllability and observability Gramians are given 
as the unique, positive semidefinite solutions of the projected generalized 
discrete-time Lyapunov equations
\begin{align} \label{eq:gdalec}
  A\mathcal{G}_{ic}A^{T} - E\mathcal{G}_{ic}E^{T} - Q_{\ell}BB^{T}Q_{\ell}^{T} 
    & = 0, ~~~~\mathcal{G}_{ic} = Q_{r}\mathcal{G}_{ic}Q_{r}^{T},\\
  \label{eq:gdaleo}
  A^{T}\mathcal{G}_{io}A - E^{T}\mathcal{G}_{io}E - Q_{r}^{T}C^{T}CQ_{r} 
    & = 0, ~~~~\mathcal{G}_{io} = Q_{\ell}^{T}\mathcal{G}_{io}Q_{\ell},
\end{align}
with $Q_{\ell} = I_{n} - P_{\ell}$ and $Q_{r} = I_{n} - P_{r}$ the spectral 
projectors onto the left and right deflating subspaces corresponding to the 
infinite eigenvalues of the matrix pencil $\lambda E - A$; see \cite{Sty02}.

Using the system Gramians, the set of Hankel singular values is defined in the 
following; see \cite{morMehS05}.

\begin{definition}
  The square roots of the $n_{f}$ largest eigenvalues of
  $\mathcal{G}_{pc}E^{T}\mathcal{G}_{po}E$ denoted by $\varsigma_{1} \geq 
  \varsigma_{2} \geq \cdots \geq \varsigma_{n_{f}}$ are the proper Hankel 
  singular values of \eqref{eq:desc}.
  The square roots of the $n_{\infty}$ largest eigenvalues of 
  $\mathcal{G}_{ic}A^{T}\mathcal{G}_{io}A$ denoted by $\theta_{1} \geq 
  \theta_{2} \geq \cdots \geq \theta_{n_{\infty}}$ are the improper Hankel 
  singular values of \eqref{eq:desc}.
\end{definition}

In case of a non-singular descriptor term $E$, the proper Hankel singular values
are the classical Hankel singular values of the system.
Therefor, an equivalent energy interpretation of the proper Hankel singular 
values exists which proposes the truncation of states corresponding to small 
proper Hankel singular values, which are difficult to control and observe.
Unfortunately, this does not hold for the improper Hankel singular values since 
these correspond to the constraints of the system.
The truncation of non-zero improper Hankel singular values may result in 
physically meaningless systems.

There exist diverse concepts of controllability and observability for
descriptor systems. For this paper, we restrict ourselves to the following
definitions; see, e.g., \cite{Sty02}.

\begin{definition} System \eqref{eq:desc} is called:
  \begin{enumerate}
    \item R-controllable if
      \begin{align*}
        \begin{aligned}
          \mathrm{rank}\begin{bmatrix} \lambda E - A, & B \end{bmatrix} & = n
            & \text{for all}~ & \lambda \in \mathbb{C}.
        \end{aligned}
      \end{align*}
    \item C-controllable if the system is R-controllable and
      \begin{align*}
        \mathrm{rank}\begin{bmatrix} E, & B \end{bmatrix} & = n.
      \end{align*}
    \item R-observable if
      \begin{align*}
        \begin{aligned}
          \mathrm{rank}\begin{bmatrix} \lambda E - A \\ C \end{bmatrix} & = n
            & \text{for all}~ & \lambda \in \mathbb{C}.
        \end{aligned}
      \end{align*}
    \item C-observable if the system is R-observable and
      \begin{align*}
        \mathrm{rank}\begin{bmatrix} E \\ C \end{bmatrix} & = n.
      \end{align*}
  \end{enumerate}
\end{definition}

The relation between these controllability, observability notions and the 
system Gramians is given in \cite[Theorem 2.3]{morSty04}.
Especially, all proper Hankel singular values are non-zero if and only
if the system is R-controllable and R-observable.

The mapping from past inputs $u_{-}:(-\infty, 0] \rightarrow \mathbb{R}^{m}$ 
to future outputs $y_{+}:(0, +\infty] \rightarrow \mathbb{R}^{p}$ is described
by the Hankel operator $y_{+} = \mathcal{H}u_{-}$.
A generalization of this operator to the case of descriptor systems can be found 
in \cite{morCaoSW15}.
The measure of the influence of past inputs on future outputs in the 
$\mathcal{L}_{2}$-norm leads to the definition of the Hankel semi-norm for
descriptor systems.

\begin{definition}
  The Hankel semi-norm of a system $G$ is given by
  \begin{align} \label{eq:hankelnorm}
    \lVert G \rVert_{H} & = \sup\limits_{u_{-} \in \mathcal{W}_{2}^{\nu - 1}
      (-\infty, 0]}{\frac{\lVert y_{+} \rVert_{\mathcal{L}_{2}}}{\lVert u_{-} 
      \rVert_{\mathcal{L}_{2}}}},
  \end{align}
  where $\mathcal{W}_{2}^{\nu - 1}(-\infty, 0]$ denotes the Sobolev space of 
  $\nu - 1$ times weakly differentiable functions w.r.t. the $\mathcal{L}_{2}$ 
  inner product on the interval $(-\infty, 0]$ and $\lVert . 
  \rVert_{\mathcal{L}_{2}}$ the $\mathcal{L}_{2}$-norm.
\end{definition}

In case of an invertible descriptor term $E$, the Hankel semi-norm 
\eqref{eq:hankelnorm} simplifies to
\begin{align*}
  \lVert G \rVert_{H} & = \varsigma_{\mathrm{max}}(G),
\end{align*}
where $\varsigma_{\mathrm{max}}(G)$ is the largest Hankel singular value of the 
system $G$.

%% file: sections/ghna.tex
\section{Generalized Hankel-Norm Approximation}
\label{sec:GHNA}

\subsection{Algorithm for Standard Systems}
First, the algorithm for the standard system case, introduced by Glover in
\cite{morGlo84}, is considered.
Therefor, a balanced minimal realization of the given standard 
system $(I_{n_{\mathrm{min}}}, A, B, C, D)$ is assumed, where
$n_{\mathrm{min}}$ is the McMillan degree of the system, i.e., the order of its
minimal realization.
The computation is usually done by the balanced truncation square root method.
Since the resulting system is balanced and minimal, the system Gramians are 
equal and diagonal
\begin{align*}
  \mathcal{G}_{pc} = \mathcal{G}_{po} = \mathrm{diag}(\varsigma_1, 
    \varsigma_2, \ldots, \varsigma_{n_{\mathrm{min}}}),
\end{align*}
with $\varsigma_1, \ldots, \varsigma_{n_{\mathrm{min}}}$ all non-zero Hankel 
singular values of the system.
Next, the system is partitioned by the order $r$ such that
\begin{align*}
  \varsigma_1 \geq \ldots \geq \varsigma_{r} > \varsigma_{r + 1} = \ldots = 
    \varsigma_{r + k + 1} > \varsigma_{r + k + 2} \geq \ldots \geq 
    \varsigma_{n_{\mathrm{min}}},
\end{align*}
with $k \geq 1$ being the multiplicity of the $(r + 1)$-st Hankel singular 
value.
The Gramians are reordered to separate the block with the $(r+1)$-st Hankel 
singular value as
\begin{align} \label{eq:partGram}
  \check{\mathcal{G}}_{pc} = \check{\mathcal{G}}_{po} = 
    \begin{bmatrix} \check{\Sigma} & \\ & \varsigma_{r + 1}I_{k} \end{bmatrix},
\end{align}
with $\check{\Sigma} = \mathrm{diag}(\varsigma_1, \ldots, \varsigma_{r}, 
\varsigma_{r + k + 2}, \ldots, \varsigma_{n_{\mathrm{min}}})$.
Accordingly to \eqref{eq:partGram}, the remaining system matrices have to be 
permuted and partitioned 
\begin{align*}
  \check{A} = \begin{bmatrix} A_{11} & A_{12} \\ A_{21} & A_{22} 
    \end{bmatrix},~ \check{B} = \begin{bmatrix} B_{1} \\ B_{2} 
    \end{bmatrix},~ \check{C} = \begin{bmatrix} C_{1}, & C_{2}
    \end{bmatrix}.
\end{align*}
Then, the partitioned system is transformed by the following formulas
\begin{align} \label{eq:transform}
  \begin{aligned}
    \tilde{A} & = \Gamma^{-1}(\varsigma_{r+1}^{2}A_{11}^{T} + 
      \check{\Sigma} A_{11} \check{\Sigma} + 
      \varsigma_{r+1}C_{1}^{T}UB_{1}^{T}),\\
    \tilde{B} & = \Gamma^{-1}(\check{\Sigma} B_{1} - 
      \varsigma_{r+1}C_{1}^{T}U),\\
    \tilde{C} & = C_{1}\check{\Sigma} - \varsigma_{r+1}UB_{1}^{T},\\
    \tilde{D} & = D + \varsigma_{r+1}U,
  \end{aligned}
\end{align}
with $\Gamma = \check{\Sigma}^{2} - \varsigma_{r+1}^{2}I_{n_{\mathrm{min}}-k}$
and $U = (C_{2}^{T})^{\dagger}B_{2}$.
Here, $M^{\dagger}$ denotes the Moore-Penrose pseudo-inverse of a matrix $M$.
This system is constructed such that the error transfer function 
$\mathcal{E} = G - \tilde{G}$ is scaled all-pass with
$\tilde{G}$ the transfer function of \eqref{eq:transform}, i.e., it holds
\begin{align} \label{eq:allpasserr}
  \mathcal{E}(s)\mathcal{E}^{T}(-s) = \varsigma_{r+1}^{2}I_{p},
\end{align}
for all $s \in \mathbb{C}$ that are not poles of $\mathcal{E}(s)$ or
$\mathcal{E}^{T}(-s)$.
In this case, the approximation error satisfies
\begin{align} \label{eq:allpasserror}
  \lVert \mathcal{E} \rVert_{H} = \lVert \mathcal{E} 
    \rVert_{\mathcal{L}_{\infty}} = \varsigma_{r+1}.
\end{align}

The transfer function $\tilde{G}$ of \eqref{eq:transform} has exactly 
$n_{\mathrm{min}} - k - r$ unstable poles.
As last step, an additive decomposition of $\tilde{G}$ is computed such
that $\tilde{G} = G_{h} + G_{+}$, where $G_{+}$ is the anti-stable part
of order $n_{\mathrm{min}} - k - r$ and $G_{h}$ is the stable part of order $r$.
Since the Hankel semi-norm only depends on the stable part of the system,
the error \eqref{eq:allpasserror} in the Hankel semi-norm does not change if the
unstable part is removed, such that
\begin{align} \label{eq:hankelerror}
  \lVert G - G_{h} \rVert_{H} = \varsigma_{r+1}.
\end{align}

\subsection{Computing a Balanced Realization for Descriptor Systems}

As for the standard system case, for descriptor systems, a balanced 
conditionally minimal realization is needed.
The term ``conditionally'' minimal means that the order of the system is minimal
except of the reduction of the index-$1$ parts in $E$, see \cite{Sok06}.
The computation is done using the generalized balanced truncation square root 
method (GBT(SR)).
The basic idea of this method is the computation of a balanced realization and
the truncation of unnecessary states.

\begin{definition}
  A realization of a descriptor system \eqref{eq:desc} is called balanced if
  \begin{align*}
    \begin{aligned}
      \mathcal{G}_{pc} & = \mathcal{G}_{po} = \begin{bmatrix} \Sigma & 0 \\ 
      0 & 0 \end{bmatrix} & \text{and} && \mathcal{G}_{ic}
      & = \mathcal{G}_{io} = \begin{bmatrix} 0 & 0 \\ 0 & \Theta \end{bmatrix}
    \end{aligned}
  \end{align*}
  hold, with the proper Hankel singular values
  $\Sigma = \mathrm{diag}(\varsigma_{1}, \ldots, \varsigma_{n_{f}})$ and the
  improper Hankel singular values $\Theta = \mathrm{diag}(\theta_{1}, \ldots,
  \theta_{n_{\infty}})$.
\end{definition}

The truncation of the states is made with respect to the computed Hankel 
singular values.
The proper Hankel singular values have the same meaning as the classical Hankel 
singular values in the standard case, i.e., states corresponding to small proper 
Hankel singular values are difficult to control and observe at the same time and
can be omitted.
In case of the improper Hankel singular values, only zeros can be truncated 
since the truncation of non-zero improper Hankel singular values results in 
physically meaningless approximations \cite{morMehS05}.
The number of non-zero improper Hankel singular values is equal to the rank of 
the matrix $\mathcal{G}_{ic}A^{T}\mathcal{G}_{io}A$, which can in fact be
bounded by
\begin{align} \label{eq:impbound}
  \mathrm{rank}(\mathcal{G}_{ic}A^{T}\mathcal{G}_{io}A) \leq
    \min(\nu m, \nu p, n_{\infty}),
\end{align}
with $\nu$, index of the system, $m$, number of inputs, $p$, number of outputs, 
and $n_{\infty}$, dimension of the deflating subspace corresponding to the 
infinite eigenvalues of $\lambda E - A$.
So for large $n_{\infty}$ and usually small $\nu$, the descriptor system
\eqref{eq:desc} can be reduced significantly by the truncation of zero improper
Hankel singular values.

One method to compute the balanced truncation of a descriptor system is the 
square root method.
Therefor, consider the skinny singular value decompositions
\begin{align} \label{eq:svd1} 
  L_{p}^{T}ER_{p} & = \begin{bmatrix} U_{1}, & U_{2} \end{bmatrix} 
    \begin{bmatrix} \Sigma_{1} & 0 \\ 0 & \Sigma_{2} \end{bmatrix} 
    \begin{bmatrix} V_{1}^{T} \\ V_{2}^{T} \end{bmatrix},\\ \label{eq:svd2}
  L_{i}^{T}AR_{i} & = U_{3} \Theta_{3} V_{3}^{T},
\end{align}
with $\mathcal{G}_{pc} = R_{p}R_{p}^{T}$, $\mathcal{G}_{po} = L_{p}L_{p}^{T}$, 
$\mathcal{G}_{ic} = R_{i}R_{i}^{T}$, and $\mathcal{G}_{io} = L_{i}L_{i}^{T}$.
The matrices $\begin{bmatrix} U_{1}, & U_{2} \end{bmatrix}$, $\begin{bmatrix} 
V_{1}, & V_{2} \end{bmatrix}$, $U_{3}$ and $V_{3}$ have orthonormal columns and 
the diagonal matrices $\Sigma_{1}$, $\Sigma_{2}$, and $\Theta_{3}$ contain the 
non-zero proper and improper Hankel singular values, respectively.
The partition of the proper Hankel singular values is chosen such that 
$\Sigma_{1}$ contains all the desired Hankel singular values and $\Sigma_{2}$ 
the undesired ones.
By using the singular value decompositions in \eqref{eq:svd1} and 
\eqref{eq:svd2}, the following transformation matrices can be defined
\begin{align} \label{eq:bttrans}
  \begin{aligned}
    W_{\ell} & = \begin{bmatrix} L_{p}U_{1}\Sigma_{1}^{-\frac{1}{2}}, & 
      L_{i}U_{3}\Theta_{3}^{-\frac{1}{2}} \end{bmatrix} \in 
      \mathbb{R}^{n \times \ell}, \\
    T_{\ell} & = \begin{bmatrix} R_{p}V_{1}\Sigma_{1}^{-\frac{1}{2}}, & 
      R_{i}V_{3}\Theta_{3}^{-\frac{1}{2}} \end{bmatrix} \in 
      \mathbb{R}^{n \times \ell},
  \end{aligned}
\end{align}
where $\ell = \ell_{f} + \ell_{\infty}$ is the sum of the number of desired 
proper Hankel singular values $\ell_{f}$ and the non-zero improper Hankel 
singular values $\ell_{\infty}$.
The transformed realization
\begin{align} \label{eq:btrel}
  (\hat{E}, \hat{A}, \hat{B}, \hat{C}, \hat{D}) = (W_{\ell}^{T}ET_{\ell}, 
    W_{\ell}^{T}AT_{\ell}, W_{\ell}^{T}B, CT_{\ell}, D)
\end{align}
is of order $\ell$ and balanced with the set of Hankel singular values contained 
in $\Sigma_{1}$ and $\Theta_{3}$.
The resulting matrix pencil $\lambda\hat{E} - \hat{A}$ is a resembling of the 
Weierstrass canonical form \eqref{eq:weierform}, such that
\begin{align} \label{eq:srweier}
  \begin{aligned}
    \hat{E} & = \begin{bmatrix} I_{\ell_{f}} & 0 \\ 0 & E_{\infty} \end{bmatrix} 
    & \text{and} && \hat{A} & = \begin{bmatrix} A_{f} & 0 \\ 0 & 
    I_{\ell_{\infty}} \end{bmatrix}
  \end{aligned}
\end{align}
hold, where $A_{f} \in \mathbb{R}^{\ell_{f} \times \ell_{f}}$ is non-singular 
and $E_{\infty} \in \mathbb{R}^{\ell_{\infty} \times \ell_{\infty}}$ is 
nilpotent with index $\nu$.

Due to the reason that only the zero improper Hankel singular values have been 
truncated, the polynomial part of the system $G$ has not changed.
So it can be shown that the same error bound as for the classical balanced 
truncation method holds.
Let $\hat{G}$ be the transformed descriptor system \eqref{eq:btrel}, then it 
holds
\begin{align*}
  \lVert G - \hat{G} \rVert_{\mathcal{H}_{\infty}} \leq 2 \sum\limits_{k = 
    \ell_{f}+1}^{n_{f}}{\varsigma_{k}(G),}
\end{align*}
with $\varsigma_{k}(G)$ the $k$-th proper Hankel singular value of $G$.

\subsection{Hankel-Norm Approximation of Descriptor Systems}

As for the standard case, the GHNA method for descriptor systems is based on the
construction of an error system with all-pass transfer function
\eqref{eq:allpasserr}.
The following theorem provides an algebraic characterization of descriptor
systems with all-pass transfer functions.

\begin{theorem} \label{thm:allpass}
  Let $(E, A, B, C, D)$ be a realization of a descriptor system \eqref{eq:desc}
  with a regular matrix pencil $\lambda E - A$, the same number of inputs and 
  outputs, $m = p$, the system's transfer function $G(s)$ and $\varsigma > 0$ a 
  real constant.
  Also, it is assumed that the descriptor system is $R$-controllable and
  $R$-observable.
  Then $G(s)$ is all-pass, i.e., $G(s)G^{T}(-s) = \varsigma^{2}I_{m}$ holds,
  if and only if the following conditions are satisfied:
  \begin{enumerate}
    \item There are symmetric matrices $\mathcal{G}_{pc}$ and $\mathcal{G}_{po}$
      with 
      \begin{align} \label{eq:cproj}
        \mathcal{G}_{pc} & = P_{r}\mathcal{G}_{pc}P_{r}^{T},\\ \label{eq:oproj}
        \mathcal{G}_{po} & = P_{\ell}^{T} \mathcal{G}_{po}P_{\ell}.
      \end{align}
    \item The matrices $\mathcal{G}_{pc}$ and $\mathcal{G}_{po}$ are the 
      solutions of the projected generalized continuous-time Lyapunov equations
      \begin{align} \label{eq:control}
        E\mathcal{G}_{pc}A^{T} + A \mathcal{G}_{pc}E^{T} + 
          P_{\ell}BB^{T}P_{\ell}^{T} & = 0,\\ \label{eq:observe}
        E^{T}\mathcal{G}_{po}A + A^{T}\mathcal{G}_{po}E +
          P_{r}^{T}C^{T}CP_{r} & = 0.
       \end{align}
     \item The proper Hankel singular values satisfy
       \begin{align} \label{eq:phsv1}
         \mathcal{G}_{pc}E^{T}\mathcal{G}_{po}E & = \varsigma^{2}P_{r},\\
         \label{eq:phsv2}
         \mathcal{G}_{po}E\mathcal{G}_{pc}E^{T} & = \varsigma^{2}P_{\ell}^{T}.
       \end{align}
     \item Let $G(s) = G_{sp}(s) + P(s)$ be decomposed into the strictly
       proper part $G_{sp}$ and the polynomial part $P$. Then it holds 
       $P(s) = \sum\limits_{k = 0}^{\infty}{M_k}s^{k}$ with
       \begin{align} \label{eq:M0}
         M_{0}M_{0}^{T} & = \varsigma^{2}I_{m},\\ \label{eq:Mk}
         M_{k} & = 0 \quad\text{for}\quad k \geq 1.
       \end{align}
     \item Also, the following constraints hold
       \begin{align} \label{eq:addcon1}
         M_{0}^{T}CP_{r} + B^{T}\mathcal{G}_{po}E & = 0,\\ \label{eq:addcon2}
         M_{0}B^{T}P_{\ell}^{T} + C\mathcal{G}_{pc}E^{T} & = 0.
       \end{align}
   \end{enumerate}
\end{theorem}

\begin{proof}
  The proof can be found in the Appendix.
\end{proof}

Theorem \ref{thm:allpass} can be used to derive more general construction
formulas for all-pass error systems, see \cite{morBenW17} for generalized 
formulas for an invertible $E$ matrix.
Also, it shows that the improper part of the system should not change.
For the development of an algorithm, the structure of the reduced-order model
\eqref{eq:srweier}, obtained from the generalized balanced truncation, can be
exploited.
So, let the matrices $\hat{B}$ and $\hat{C}$ be partitioned accordingly to 
\eqref{eq:srweier} as
\begin{align*}
  \begin{aligned}
    \hat{B} & = \begin{bmatrix} B_{f} \\ B_{\infty} \end{bmatrix} & \text{and} &&
      \hat{C} & = \begin{bmatrix} B_{f}, & B_{\infty} \end{bmatrix}.
  \end{aligned}
\end{align*}
Using this block partition, the system $(\hat{E}, \hat{A}, \hat{B}, \hat{C}, D)$ 
automatically decouples into its slow subsystem
\begin{align} \label{eq:slow}
  \begin{aligned}
    \dot{x}_{f}(t) & = A_{f}x_{f}(t) + B_{f}u(t),\\
    y_{f}(t) & = C_{f}x_{f}(t)
  \end{aligned}
\end{align}
and its fast subsystem
\begin{align} \label{eq:fast}
  \begin{aligned}
    E_{\infty}\dot{x}_{\infty}(t) & = x_{\infty}(t) + B_{\infty}u(t),\\
    y_{\infty}(t) & = C_{\infty}x_{\infty}(t) + Du(t).
  \end{aligned}
\end{align}

First, the fast subsystem \eqref{eq:fast} is considered.
Since the GBT(SR) was used to compute the system \eqref{eq:fast}, there are no 
zero improper Hankel singular values anymore.
As mentioned before and considering Theorem \ref{thm:allpass}, there is no 
meaningful further reduction concerning the improper Hankel singular values 
without generating physically meaningless results, so the fast subsystem stays 
unchanged.

Now, let us consider the slow subsystem \eqref{eq:slow}.
It is easy to see that \eqref{eq:slow} is in standard form.
Also beneficial properties, resulting from the applied balanced truncation 
method, still hold for this subsystem which means it is stable and balanced.

Let the original system be decomposed as $G = G_{sp} + P$ into its slow
subsystem $G_{sp}$ and its fast subsystem $P$.
By the truncation of only zero proper Hankel singular values, the system 
\eqref{eq:slow} is a minimal realization of the original slow subsystem 
$G_{sp}$.
Now, the standard HNA method, mentioned in the previous section, can be applied 
to \eqref{eq:slow}.
As result, an $r$-th order HNA is computed
\begin{align} \label{eq:hankelsys}
  \begin{aligned}
    E_{h}\dot{x}_{h}(t) & = A_{h}x_{h}(t) + B_{h}u(t),\\
    y_{h}(t) & = C_{h}x_{h}(t) + D_{h}u(t),
  \end{aligned}
\end{align}
where $E_{h}$ results from avoiding disadvantageous scaling of the matrices
$A_{h}$ and $B_{h}$.
More general transformation formulas for invertible $E$ matrices have been 
developed in \cite{morBenW17}.
To get an optimal HNA of the descriptor system \eqref{eq:desc}, now the computed 
HNA \eqref{eq:hankelsys} and the reduced-order fast subsystem 
\eqref{eq:fast} are coupled
\begin{align}\label{eq:redsys}
  \begin{aligned}
    \begin{bmatrix} E_{h} & 0 \\ 0 & E_{\infty} \end{bmatrix} \dot{\hat{x}}(t) 
      & = \begin{bmatrix} A_{h} & 0 \\ 0 & I_{\ell_{\infty}}\end{bmatrix} 
      \hat{x}(t) + \begin{bmatrix} B_{h} \\ B_{\infty}\end{bmatrix}u(t), \\
    \hat{y}(t) & = \begin{bmatrix} C_{h}, & C_{\infty}\end{bmatrix}\hat{x}(t) 
      + (D_{h} + D)u(t).
  \end{aligned}
\end{align}
In the following theorem, the properties of the resulting GHNA are summarized.

\begin{theorem}
  Let $G$ be a c-stable descriptor system \eqref{eq:desc} with a regular matrix 
  pencil.
  The $\ell$-th order generalized Hankel-norm approximation \eqref{eq:redsys}, 
  with its transfer function $\hat{G}$ and $\ell = r + \ell_{\infty}$, has the 
  following properties:
  \begin{enumerate}
    \item \label{thm:ghna1}
      The realization of $\hat{G}$ is conditionally minimal and c-stable.
    \item \label{thm:ghna2}
      The absolute error in the Hankel semi-norm is given by
      \begin{align*}
        \lVert G - \hat{G} \rVert_{H} & = \varsigma_{r+1}(G),
      \end{align*}
      where $\varsigma_{r+1}(G)$ is the $(r+1)$-st proper Hankel singular value 
      of $G$.
    \item \label{thm:ghna3}
      The absolute error in the $\mathcal{H}_{\infty}$-norm can be bounded by
      \begin{align*}
        \lVert G - \hat{G} \rVert_{\mathcal{H}_{\infty}} & \leq 
          2\sum\limits_{k = r+1}^{n_{f}}{\varsigma_{k}(G)},
      \end{align*}
      where $\varsigma_{k}(G)$ is the $k$-th proper Hankel singular values 
      of $G$.
  \end{enumerate}
\end{theorem}

\begin{proof}
  Let $G = G_{sp} + P$ be the original system and $\tilde{G} = G_{b} + P_{b}$ 
  the balanced, conditionally minimal realization obtained by the GBT(SR)
  method.
  Here $G_{sp}$, $G_{b}$ denote the slow subsystems and $P$, $P_{b}$ the fast 
  ones.
  The GHNA is constructed by
  \begin{align} \label{eq:hankeldec}
    \hat{G} & = G_{h} + P_{b},
  \end{align}
  where $G_{h}$ is the $r$-th order HNA \eqref{eq:hankelsys} of the standard system 
  $G_{b}$.
  
  First, consider part \ref{thm:ghna1}.
  The balanced realization $\tilde{G}$ is conditionally minimal and c-stable.
  So by construction \eqref{eq:hankeldec}, both of these properties are 
  transferred to the GHNA.
  
  Now consider the error formulas in \ref{thm:ghna2}. and \ref{thm:ghna3}.
  Therefor, let $\mathcal{E} = G - \hat{G}$ be the error system of the GHNA.
  Then it holds
  \begin{align*}
    \mathcal{E} & = G - \hat{G}\\
    & = G_{sp} + P - G_{h} - P_{b}\\
    & = G_{b} + P_{b} - G_{h} - P_{b}\\
    & = G_{b} - G_{h},
  \end{align*}
  since the balanced realization $\tilde{G}$ is conditionally minimal and 
  therefor, $G_{b} = G_{sp}$ and $P_{b} = P$.
  Using the error bound of the standard method \eqref{eq:hankelerror} one 
  obtains
  \begin{align*}
    \lVert G - \hat{G} \rVert_{H} = \lVert G_{b} - G_{h} \rVert_{H} 
      = \varsigma_{r+1}(G_{b}) = \varsigma_{r+1}(G).
  \end{align*}
  Using the same approach, the error in the $\mathcal{H}_{\infty}$-norm is 
  given by
  \begin{align*}
    \lVert G - \hat{G} \rVert_{\mathcal{H}_{\infty}} = \lVert G_{b} - G_{h}
      \rVert_{\mathcal{H}_{\infty}} \leq 2\sum\limits_{k = r+1}^{n_{f}}
      {\varsigma_{k}(G)},
  \end{align*}
  if the $\mathcal{H}_{\infty}$-norm error bound for the standard $r$-th order 
  HNA from \cite{morAnt05} is used.
\end{proof}
 
In Algorithm \ref{alg:ghna}, the complete GHNA method is summarized.
\vspace{1em}

\begin{algorithm}
  \caption{Generalized Hankel-Norm Approximation (GHNA) Method}
  \label{alg:ghna}
  \begin{algorithmic}[1]
    \State Solve the continuous-time Lyapunov equations \eqref{eq:gcalec} and 
      \eqref{eq:gcaleo} for the Cholesky factorizations $\mathcal{G}_{pc} = 
      R_{p}R_{p}^{T}$ and $\mathcal{G}_{po} = L_{p}L_{p}^{T}$.
    \State Solve the discrete-time Lyapunov equations \eqref{eq:gdalec} and 
      \eqref{eq:gdaleo} for the Cholesky factorizations $\mathcal{G}_{ic} = 
      R_{i}R_{i}^{T}$ and $\mathcal{G}_{io} = L_{i}L_{i}^{T}$.
    \State Compute the two skinny singular value decompositions
      \begin{align*}
        \begin{aligned}
          L_{p}^{T}ER_{p} & = U_{1} \Sigma V_{1}^{T} & \text{and} &&
            L_{i}^{T}AR_{i} & = U_{2} \Theta V_{2}^{T}.
          \end{aligned}
      \end{align*}
    \State Compute the transformation matrices
      \begin{align*}
        \begin{aligned}
        W_{p} & = L_{p}U_{1}\Sigma^{-\frac{1}{2}}, & T_{p} & =
          R_{p}V_{1}\Sigma^{-\frac{1}{2}}, \\
        W_{i} & = L_{i}U_{2}\Theta^{-\frac{1}{2}}, & T_{i} & =
          R_{i}V_{2}\Theta^{-\frac{1}{2}}.
        \end{aligned}
      \end{align*}
    \State Compute the minimal balanced realization of the slow subsystem
      \begin{align*}
        (I_{\ell_{f}}, A_{f}, B_{f}, C_{f}, 0) & = (W_{p}^{T}ET_{p}, 
          W_{p}^{T}AT_{p}, W_{p}^{T}B, CT_{p}, 0).
      \end{align*}
    \State Choose the proper Hankel singular value $\varsigma_{r + 1}$.
    \State Permute and partition the Gramians of the slow subsystem
      \begin{align*}
        \check{\mathcal{G}}_{pc} = \check{\mathcal{G}}_{po} =
          \mathrm{diag}(\check{\Sigma}, \varsigma_{r + 1}I_{k}),
      \end{align*}
      and the corresponding system matrices
      \begin{align*}
        \begin{aligned}
          \check{A} & = \begin{bmatrix} A_{11} & A_{12} \\ A_{21} & A_{22} 
            \end{bmatrix}, & \check{B} & = \begin{bmatrix} B_{1} \\ B_{2} 
            \end{bmatrix}, & \check{C} & = \begin{bmatrix} C_{1}, & C_{2}
            \end{bmatrix}.
        \end{aligned}
      \end{align*}
    \State Compute the all-pass transformation
      \begin{align*}
        \tilde{E} & = \check{\Sigma}^{2} - \varsigma_{r+1}^{2}I_{\ell_{f}-k},\\
        \tilde{A} & = \varsigma_{r+1}^{2}A_{11}^{T} + 
          \check{\Sigma}A_{11}\check{\Sigma} + 
          \varsigma_{r+1}C_{1}^{T}UB_{1}^{T},\\
        \tilde{B} & = \check{\Sigma}B_{1} - 
          \varsigma_{r+1}C_{1}^{T}U,\\
        \tilde{C} & = C_{1}\check{\Sigma} - \varsigma_{r+1}UB_{1}^{T},\\
        \tilde{D} & = \varsigma_{r+1}U,
      \end{align*}
      with $U = (C_{2}^{T})^{\dagger}B_{2}$.
    \State Compute the additive decomposition
      \begin{align*}
        \tilde{G}(s) = \tilde{C}(s\tilde{E} - \tilde{A})\tilde{B} + 
          \tilde{D} = G_{h}(s) + F(s),
      \end{align*}
      where $F$ is anti-stable and $G_{h}$ stable with the realization 
      $(E_{h}, A_{h}, B_{h}, C_{h}, D_{h})$.
    \State Compute the balanced realization of the fast subsystem
      \begin{align*}
        (E_{\infty}, I_{\ell_{\infty}}, B_{\infty}, C_{\infty}, D) & = 
          (W_{i}^{T}ET_{i}, W_{i}^{T}AT_{i}, W_{i}^{T}B, CT_{i}, D).
      \end{align*}
    \State Couple the resulting subsystems
      \begin{align*}
        (\hat{E}, \hat{A}, \hat{B}, \hat{C}, \hat{D}) & = 
          \left(\begin{bmatrix} E_{h} & 0 \\ 0 & E_{\infty} 
          \end{bmatrix}, \begin{bmatrix} A_{h} & 0 \\ 0 & I_{\ell_{\infty}}
          \end{bmatrix}, \begin{bmatrix} B_{h} \\ B_{\infty}\end{bmatrix}, 
          \begin{bmatrix} C_{h}, & C_{\infty}\end{bmatrix}, D_{h} + D \right).
      \end{align*}
  \end{algorithmic}
\end{algorithm}

%% file: sections/methods.tex
\section{Numerical Methods for GHNA}
\label{sec:methods}

\subsection{Approximate GHNA}

The GHNA method can quickly become numerically unstable.
This problem arises from the transformation formulas \eqref{eq:transform} 
for the construction of a scaled all-pass error transfer function.
It is easy to see that the diagonal matrix 
$\Gamma = \check{\Sigma}^{2} - \varsigma_{r+1}^{2}I_{n_{\min}-k}$ can
lead to large numerical errors for small proper Hankel singular values in
further computations.
This happens if either the chosen value $\varsigma_{r+1}$ or the 
remaining proper Hankel singular values in $\check{\Sigma}$ are very small.
One preventive measure was the usage of the descriptor system structure 
\eqref{eq:hankelsys} to avoid unnecessary scaling by $\Gamma$. 
In further considerations, only the case of too small remaining Hankel singular
values is treated.

Small proper Hankel singular values can arise from numerical errors during the 
computation of the minimal realization.
Therefor, one approach to solve this problem is to compute a smaller balanced 
truncation of the slow subsystem than the minimal realization such that
too small proper Hankel singular values are cut off.
In this case, an additional error is made since the balanced realization is only 
an approximation of the original system.
To get a measure for the additional error, let $G_{b}$ be the computed balanced 
truncation of order $n_{b}$ of the slow subsystem $G_{sp}$.
Then it has been shown in \cite{morGlo84} that in the Hankel semi-norm it holds
\begin{align} \label{eq:hankelbterror}
  \lVert G_{sp} - G_{b} \rVert_{H} & \leq 2\sum\limits_{k = n_{b} + 1}^{n_{f}}
    {\varsigma_{k}(G_{sp})},
\end{align}
with $n_{f}$ the order of the slow subsystem $G_{sp}$.
For the overall error, let $G = G_{sp} + P$ be the original descriptor system 
and $\tilde{G} = G_{b} + P_{b}$ the balanced realization with $G_{b}$ of order 
$n_{b}$.
The generalized Hankel-norm approximation is denoted by $\hat{G} = G_{h} + 
P_{b}$, where the $r$-th order standard Hankel-norm approximation $G_{h}$ was 
computed from the balanced realization $G_{b}$.
Using \eqref{eq:hankelbterror} one obtains
\begin{align} \label{eq:apprerror}
  \begin{aligned}
    \lVert G - \hat{G} \rVert_{H} & = \lVert G_{sp} + P - G_{h} - P_{b} 
      \rVert_{H}\\
    & = \lVert G_{sp} - G_{h} \rVert_{H}\\
    & = \lVert G_{sp} - G_{b} + G_{b} - G_{h} \rVert_{H} \\
    & \leq \lVert G_{b} - G_{h} \rVert_{H} + \lVert G_{sp} - G_{b} 
      \rVert_{H} \\
    & \leq \varsigma_{r + 1}(G_{b}) + 2\sum\limits_{k = n_{b} + 1}^{n_{f}}
      {\varsigma_{k}(G_{sp})}.
  \end{aligned}
\end{align}
Since balancing the system does not change the Hankel singular values, the 
Hankel singular values of $G_{b}$ and $G_{sp}$ are also the proper 
Hankel singular values of $G$.
The resulting error can be bounded by
\begin{align*}
  \lVert G - \hat{G} \rVert_{H} & \leq \varsigma_{r + 1}(G) + 
    2\sum\limits_{k = n_{b} + 1}^{n_{f}}{\varsigma_{k}(G)}.
\end{align*}
Concerning the $\mathcal{H}_{\infty}$-norm, the approach in \eqref{eq:apprerror} 
can be used to get
\begin{align*}
  \lVert G - \hat{G} \rVert_{\mathcal{H}_{\infty}} & \leq
    2\sum\limits_{k = r + 1}^{n_{f}}{\varsigma_{k}(G)},
\end{align*}
which is the same error bound as for the exact method.

This approximate version of the GHNA takes advantage of the use of the GBT(SR) 
method in form of the adaptive choice of the order $n_{b}$.
It is possible to choose the order $n_{b}$ with respect to the proper Hankel 
singular value $\varsigma_{r + 1}$ such that
\begin{align*}
  2\sum\limits_{k = n_{b} + 1}^{n_{f}}{\varsigma_{k}(G)} & \ll
    \varsigma_{r + 1}(G).
\end{align*}
In this case, the resulting additional error becomes negligible small concerning
the original Hankel semi-norm error.
But the corresponding matrix $\Gamma$ leads to a better conditioned problem.
The algorithmic adjustments in the implementation of the GHNA method are small,
since only the truncation of non-zero proper Hankel singular values has to be
allowed in the generalized balanced truncation method.
In this case, the $\Sigma_{2}$ term in \eqref{eq:svd1} with the undesired proper 
Hankel singular values is not zero and only the matrices $U_{1}$, $\Sigma_{1}$,
and $V_{1}$ are used for further computations.

Another advantage of the approximate algorithm can be found in the
computation of the balanced truncation.
The GBT(SR) method needs to scale the transformation matrices \eqref{eq:bttrans}
using the inverse remaining Hankel singular values which is more accurate if the
small proper Hankel singular values are truncated.
Also in the sense of computational costs, this approximate method has
advantages.
The further steps of the algorithm, i.e., the all-pass transformation and 
additive decomposition, are extremely costly for large matrices in terms of
computational time and memory usage.
Therefor, it is advantageous to already have a small balanced realization 
for the further computations.

\subsection{Application to Sparse Systems}
\label{sec:sparse}

A frequently appearing case in practice is the model reduction of large-scale 
sparse descriptor systems.
In this case, the system matrices $E$ and $A$ from the descriptor system
\eqref{eq:desc} are in a large-scale sparse form, i.e., the dimension $n$ is
large, the matrices can be stored using $\mathcal{O}(n)$ memory and the
matrix-vector multiplication can be computed in $\mathcal{O}(n)$ effort.
Often such matrices result from the discretization of partial differential 
equations.

The transformation into a balanced realization does not preserve the
sparsity of the system matrices.
Therefor, the GHNA method can only be adapted to sparse systems in the first two
steps.
This concerns the computation of the solutions of the generalized projected
Lyapunov equations \eqref{eq:gcalec}--\eqref{eq:gdaleo}.
It has been observed that the eigenvalues of the symmetric positive semidefinite 
solutions of Lyapunov equations with low-rank right-hand sides generally decay 
rapidly.
The same result holds for the generalized projected Lyapunov equations
\cite{Sty08}.
Therefor, the system Gramians can be approximated by low-rank Cholesky 
factorizations, e.g., $\mathcal{G}_{pc} \approx Z_{pc}Z_{pc}^{T}$ with $Z_{pc} 
\in \mathbb{R}^{n \times k}$ and $k \ll n$.

For the proper system Gramians, the computation is done by adapting existing
low-rank methods, e.g., Krylov subspace methods or low-rank ADI methods.
In this case, the right-hand side has to be replaced by the projected form
from the Lyapunov equations \eqref{eq:gcalec}, \eqref{eq:gcaleo}.
Additionally, it is recommended to project the solution back into the
corresponding subspace after some steps of the methods due to a drift-off effect.

In contrast to this, for the improper system Gramians full-rank factorizations
can be constructed explicitly such that $G_{ic} = Z_{ic}Z_{ic}^{T}$ and
$G_{io} = Z_{io}Z_{io}^{T}$, with
\begin{align*}
  Z_{ic} & =\begin{bmatrix} Q_{r}A^{-1}B, & A^{-1}EQ_{r}A^{-1}B, & 
    \ldots, & (A^{-1}E)^{\nu - 1}Q_{r}A^{-1}B \end{bmatrix},\\
  Z_{io} & = \begin{bmatrix} Q_{\ell}^{T}A^{-T}C^{T}, & A^{-T}E^{T}
    Q_{\ell}^{T}A^{-T}C^{T}, & \ldots, & (A^{-T}E^{T})^{\nu - 1}
    Q_{\ell}^{T}A^{-T}C^{T} \end{bmatrix};
\end{align*}
see \cite{Sty08} for more details.
Thereby, the size of the full-rank factorizations is bounded by the number of
inputs $m$ or outputs $p$ times the system's index $\nu$.
This corresponds to the overall bound of the non-zero improper Hankel singular
values \eqref{eq:impbound}.

Still for using these methods, the spectral projections $P_{\ell}$, $P_{r}$,
$Q_{\ell}$ and $Q_{r}$ have to be computed. But for many problems, these spectral
projections can be applied by exploiting the special structure of the
problem; see \cite{Sty08} for some examples.

\subsection{The Projection-Free Approach}
\label{sec:projfree}

In case of unstructured problems, there are no explicit construction formulas for
the spectral projectors $P_{\ell}$, $P_{r}$, $Q_{\ell}$ and $Q_{r}$, so they have
to be explicitly computed for the use in the generalized projected Lyapunov
equations \eqref{eq:gcalec}--\eqref{eq:gdaleo}.
But as for the GBT(SR) method, an alternative approach to the 
use of spectral projectors can be given; see \cite{morSty04}.

As already used in the GHNA algorithm, the GBT method can be interpreted as a 
decoupling of the original system into the slow and fast subsystems and the 
individual reduction of both.
Therefor, consider the following generalized block triangular form.
There are orthogonal matrices $U, V \in \mathbb{R}^{n \times n}$ such that
\begin{align*}
  \begin{aligned}
    E & = V \begin{bmatrix} E_{f} & E_{u} \\ 0 & E_{\infty} \end{bmatrix} 
      U^{T} & \text{and} && A & = V \begin{bmatrix} A_{f} & A_{u} \\ 0 & 
      A_{\infty}\end{bmatrix} U^{T},
  \end{aligned}
\end{align*}
where the matrix pencil $\lambda E_{f} - A_{f}$ contains all the finite 
eigenvalues of $\lambda E - A$ and the matrix pencil $\lambda E_{\infty} -
A_{\infty}$ has only infinite eigenvalues.
For the computation of a block diagonalization of the system, the coupled
Sylvester equations
\begin{align} \label{eq:coupsylv}
  E_{f}Y - ZE_{\infty} & = -{E_{u}},\\
  A_{f}Y - ZA_{\infty} & = -{A_{u}},
\end{align}
have to be solved for $Y$ and $Z$; see \cite{morBenQQ05}.
Using all of these matrices for the restricted system equivalence transformation
\begin{align*}
  \begin{aligned}
    W_{dec} & = V \begin{bmatrix} I_{n_{f}} & 0 \\ -Z^{T} & I_{n_{\infty}} 
      \end{bmatrix}, &
    T_{dec} & = U \begin{bmatrix} I_{n_{f}} & Y \\ 0 & I_{n_{\infty}}
      \end{bmatrix}
  \end{aligned}
\end{align*}
of the original descriptor system \eqref{eq:desc}, one obtains
\begin{align} \label{eq:diagdesc}
  \begin{aligned}
    \begin{bmatrix} E_{f} & 0 \\ 0 & E_{\infty} \end{bmatrix} 
      \dot{\tilde{x}}(t) & = \begin{bmatrix} A_{f} & 0 \\ 0 & A_{\infty} 
      \end{bmatrix}\tilde{x}(t) + \begin{bmatrix} B_{f} \\ B_{\infty} 
      \end{bmatrix}u(t), \\
    y(t) & = \begin{bmatrix} C_{f} & C_{\infty} \end{bmatrix}\tilde{x}(t) 
      + Du(t),
  \end{aligned}
\end{align}
where the remaining matrices are constructed as
\begin{align} \label{eq:blkright}
  \begin{aligned} 
    V^{T}B & = \begin{bmatrix} B_{u} \\ B_{\infty} \end{bmatrix}, &
      B_{f} & = B_{u} - ZB_{\infty},\\
    CU & = \begin{bmatrix} C_{f} \\ C_{u} \end{bmatrix}, &
      C_{\infty} & = C_{f}Y + C_{u}.
  \end{aligned}
\end{align}
Obviously, the realization in \eqref{eq:diagdesc} decouples into the fast and 
slow subsystems of \eqref{eq:desc}.
Since the spectral projectors of the subsystems are identity matrices, the 
corresponding Lyapunov equations \eqref{eq:gcalec}--\eqref{eq:gdaleo} simplify to
\begin{align*}
  E_{f}X_{pc}A_{f}^{T} + A_{f}X_{pc}E_{f}^{T} + B_{f}B_{f}^{T} & = 0,\\
  E_{f}^{T}X_{po}A_{f} + A_{f}^{T}X_{po}E_{f} + C_{f}^{T}C_{f} & = 0,
\end{align*}
for the slow subsystem and
\begin{align*}
  A_{\infty}X_{ic}A_{\infty}^{T} - E_{\infty}X_{ic}E_{\infty}^{T} - 
    B_{\infty}B_{\infty}^{T} & = 0,\\
  A_{\infty}^{T}X_{io}A_{\infty} - E_{\infty}^{T}X_{io}E_{\infty} - 
    C_{\infty}^{T}C_{\infty} & = 0,
\end{align*}
for the fast subsystem.
These Lyapunov equations can be computed without the spectral projections.
The matrices $X_{pc}$ and $X_{po}$ correspond to the parts of the proper
controllability and observability Gramians, which contain the potentially
non-zero proper Hankel singular values.
The same holds for $X_{ic}$, $X_{io}$ and the improper system Gramians.
For the rest of the algorithm, only the transformations have to be restricted to
the subsystems.

The projection-free approach is implemented in the version 3.0 of the MORLAB
toolbox; see \cite{morBenW17a}.
In this special implementation, the block diagonalization of the system is done 
by using a block transformation approach based on the following generalization 
of Theorem 4.1 from \cite{KagV92}.

\begin{theorem} \label{thm:adtf}
  Let $\Gamma \subset \mathbb{C}$ be a region in the complex plane which
  contains $n_{1}$ eigenvalues of the matrix pencil $\lambda E - A$.
  Let $Q, Z \in \mathbb{R}^{n \times n}$ be orthogonal matrices that
  transform the matrix pencil $\lambda E - A$ into the upper block triangular
  form:
  \begin{align*}
    Q^{T}(\lambda E - A)Z = \begin{bmatrix}Q_{1}^{T} \\ Q_{2}^{T}\end{bmatrix}
      (\lambda E - A)\begin{bmatrix}Z_{1}, Z_{2} \end{bmatrix} = 
      \begin{bmatrix}\lambda E_{11}^{(1)} - A_{11}^{(1)} & 
      \lambda E_{12}^{(1)} - A_{12}^{(1)} \\ 0 & \lambda E_{22}^{(1)} - 
      A_{22}^{(1)}\end{bmatrix},
  \end{align*}
  with $\Lambda(A_{11}^{(1)}, E_{11}^{(1)}) \subseteq \Gamma$ and 
  $\Lambda(A_{11}^{(1)}, E_{11}^{(1)}) \cap 
  \Lambda(A_{22}^{(1)}, E_{22}^{(1)}) = \emptyset$. 
  Similarly, let $U, V \in \mathbb{R}^{n \times n}$ be orthogonal matrices that
  transform the matrix pencil $\lambda E - A$ into the upper block triangular
  form:
  \begin{align*}
    U^{T}(\lambda E - A)V = \begin{bmatrix}U_{1}^{T} \\ 
      U_{2}^{T}\end{bmatrix}(\lambda E - A)\begin{bmatrix}V_{1}, V_{2} 
      \end{bmatrix} = \begin{bmatrix}\lambda E_{11}^{(2)} - A_{11}^{(2)} & 
      \lambda E_{12}^{(2)} - A_{12}^{(2)} \\ 0 & \lambda E_{22}^{(2)} - 
      A_{22}^{(2)}\end{bmatrix},
  \end{align*}
  with $\Lambda(A_{22}^{(2)}, E_{22}^{(2)}) \subseteq \Gamma$ and
  $\Lambda(A_{11}^{(2)}, E_{11}^{(2)}) \cap \Lambda(A_{22}^{(2)}, 
  E_{22}^{(2)}) = \emptyset$.
  Then
  \begin{align*}
    \begin{aligned}
      X & = \begin{bmatrix}U_{2}, & Q_{2} \end{bmatrix} && \mathrm{and} &
      Y & = \begin{bmatrix}Z_{1}, & V_{1} \end{bmatrix}
    \end{aligned}
  \end{align*}
  are transformation matrices, such that $X^{T}(\lambda E - A)Y$ has a block
  diagonal structure where the upper block contains the $n_{1}$ eigenvalues
  lying inside $\Gamma$ and the lower block has the remaining $n - n_{1}$
  eigenvalues of $\lambda E - A$ outside of $\Gamma$.
\end{theorem}

\begin{proof}
  The proof can be found in \cite[Section~5.2]{morWer16}.
\end{proof}

In contrast to the approach above, it is not necessary to compute the solution
of the coupled Sylvester equations and, due to the block orthogonal structure
of the transformation matrices, the right-hand sides are usually better
conditioned than \eqref{eq:blkright}.
In MORLAB, the right matrix pencil disk function method is used to generate the
block transformation matrices, see \cite{morWer16} for more details on the
implementation.
Additionally, Theorem \ref{thm:adtf} can be used to compute the additive
decomposition in step 9 of Algorithm \ref{alg:ghna} by separating the eigenvalues
with negative and positive real-parts.

%% file: sections/examples.tex
\section{Numerical Examples}
\label{sec:example}

Two examples have been chosen to demonstrate the introduced GHNA method.
All the computations were done on a machine with one Intel(R) Core(TM) i7-6700
CPU processor running at 3.40GHz and equipped with 8 GB total main memory. 
The computer is running on Ubuntu 16.04.4 LTS and uses MATLAB 9.1.0.441655 
(R2016b).

\subsection{Semi-Discretized Stokes Equation}

First, the method is tested on a large-scale sparse example.
The Stokes equation describes the flow of fluids at very low velocities without 
convection and coincides with the linearization of the Navier-Stokes equation 
around the zero-state.
The spatial discretization of the Stokes equation by the finite volume method 
leads to a descriptor system of the form
\begin{align} \label{eq:discstokes}
  \begin{aligned}
    \dot{v}_{h}(t) & = A_{11}v_{h}(t) + A_{12}p_{h}(t) + B_{1}u(t),\\
    0 & = A_{12}^{T}v_{h}(t) \phantom{{}+ A_{12}p_{h}(t)}+ B_{2}u(t),\\
    y(t) & = C_{1}v_{h}(t) + C_{2}p_{h}(t),
  \end{aligned}
\end{align}
where $v_{h}$ and $p_{h}$ are the semi-discretized vectors of velocity and 
pressure, respectively, and the matrices $B_{1}$, $B_{2}$, $C_{1}$, $C_{2}$ are 
all vectors.
For matrix pencils like in \eqref{eq:discstokes} the spectral projectors 
$P_{\ell}$ and $P_{r}$ are given by explicit construction formulas
\begin{align*}
  P_{\ell} & = \begin{bmatrix} \Pi & -\Pi A_{11}A_{12}(A_{12}^{T}A_{12})^{-1} 
    \\ 0 & 0 \end{bmatrix},\\
  P_{r} & = \begin{bmatrix} \Pi & 0 \\ -(A_{12}^{T}A_{12})^{-1}A_{12}^{T}
    A_{11} \Pi & 0 \end{bmatrix},
\end{align*}
where $\Pi = I_{n_{v}} - A_{12}(A_{12}^{T}A_{12})^{-1}A_{12}^{T}$ is the 
orthogonal projector onto the kernel of $A_{12}^{T}$ along the image of 
$A_{12}$; see \cite{morSty06}.
The generation of data is based on the test example 3.3 in \cite{morSch07}.
The Stokes equation was discretized on a uniform staggered grid of 
$80 \times 80$ points which leads to a descriptor system of the size
$n = 19,039$, where the matrix pencil $\lambda E - A$ has $n_{f} = 6,241$ 
finite and $n_{\infty} = 12,798$ infinite eigenvalues.
The data was generated to get a full-rank $A_{12}$ such that the system 
\eqref{eq:discstokes} is of index $2$.

\begin{figure}[t]
  \centering
  \input{graphics/stokes_conv.tikz}
  \caption{Convergence of the low-rank ADI iterations for the low-rank
    factorizations $\mathcal{G}_{pc} = Z_{pc}Z_{pc}^{T}$ and
    $\mathcal{G}_{po} = Z_{po}Z_{po}^{T}$ (Stokes example).}
  \label{fig:stokes_conv}
\end{figure}
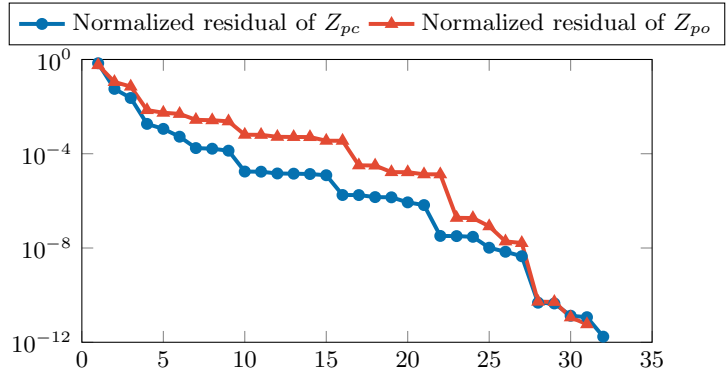

For the computation, the implementation of the GHNA method was adjusted to the 
sparse system case, as described in Section \ref{sec:sparse}, and for the
solution of the projected continuous-time Lyapunov equations \eqref{eq:gcalec}
and \eqref{eq:gcaleo}, the solvers from version 1.0.1 of the M-M.E.S.S.
toolbox have been used \cite{SaaKB16-mmess-1.0.1}.
See the demo file \texttt{bt\_mor\_DAE2.m} in \cite{SaaKB16-mmess-1.0.1}
for the used parameter settings.
With these adjusted solvers, the two iterations for the low-rank factors 
quickly converged after $31$ and $32$ iteration steps as shown in Figure
\ref{fig:stokes_conv}.
An approximation of the non-zero proper Hankel singular values has been computed
and plotted in Figure \ref{fig:stokes_hsvp} using the low-rank factorizations of 
the proper system Gramians.

\begin{figure}[t]
  \centering
  \input{graphics/stokes_hsvp.tikz}
  \caption{Computed proper Hankel singular values and the tolerance for the 
    balanced realization (Stokes example).}
  \label{fig:stokes_hsvp}
\end{figure}
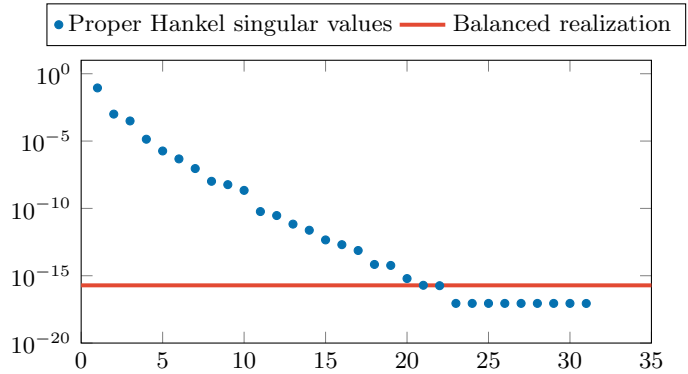

As mentioned before, it is numerically more stable using a balanced truncation 
of the slow subsystem than the minimal realization.
For this reason, a tolerance for the allowed proper Hankel singular values was 
computed as $\mathrm{log}(n) \cdot \epsilon$ and multiplied with the largest
proper Hankel singular value, with $n$ the order of the system and $\epsilon$
the machine epsilon.
The resulting bound is also shown in Figure \ref{fig:stokes_hsvp} and the 
computed balanced realization is of order $21$.

To compute a fourth order standard Hankel-norm approximation of the slow 
subsystem, the fifth proper Hankel singular value $\varsigma_{5} = 1.8370 \cdot 
10^{-6}$ was chosen.
The additive decomposition of the transformed realization \eqref{eq:transform}
was made by using the \texttt{ml\_adtf\_dss} routine from version 3.0 of the
MORLAB toolbox \cite{morBenW17a}.
The projected generalized discrete-time Lyapunov equations \eqref{eq:gdalec} 
and \eqref{eq:gdaleo} were constructed as shown in Section \ref{sec:sparse}.
In contrast to the con\-tin\-u\-ous-time case, every iteration step was 
reprojected since the iteration converges after $2$ steps at maximum.
As result only one non-zero improper Hankel singular value $\theta_{1} = 5.3046 
\cdot 10^{-18}$ was computed.
This implies that the reduced-order system would be of index $1$.
In this case, the fast subsystem \eqref{eq:fast} is equivalent to a feed-through
term of the form $-C_{\infty}B_{\infty} = -1.875 \cdot 10^{-17}.$
Since this value is negligible small compared to the resulting feed-through term
$\hat{D} = \varsigma_{5}$ from the GHNA method, the state corresponding to this 
improper Hankel singular value was truncated, too.

\begin{figure}[t]
  \centering
  \input{graphics/stokes_r5_error.tikz}
  \caption{Absolute error of the GHNA and GBT(SR) transfer functions in the 
    spectral norm with the corresponding error bound (Stokes example).}
  \label{fig:stokes_r5_error}
\end{figure}
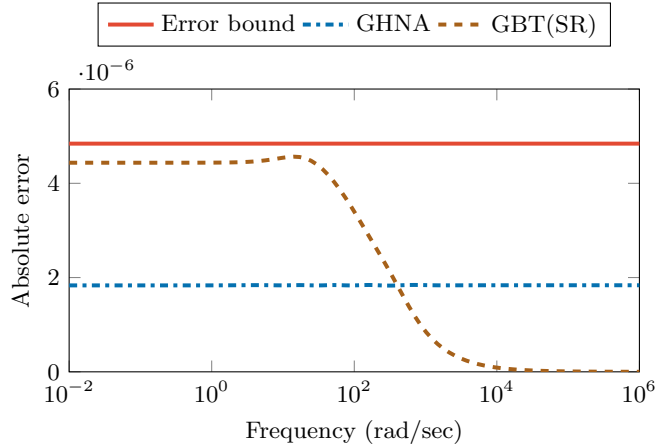

To sum up, the original semi-discretized Stokes equation is approximated by a
GHNA of order $4$ ($r = 4$, $\ell_{\infty} = 0$).
The error of the original and reduced-order transfer functions in the spectral
norm can be seen in Figure \ref{fig:stokes_r5_error}.
Additionally, the corresponding $\mathcal{H}_{\infty}$ error bound as well as a 
reduced-order model of the same order computed by the GBT(SR) method are plotted
to get an impression of the approximation behavior.
The shown error behavior of the GHNA is typical.
Since the reduced-order model is based on an all-pass error transfer function, 
the error behavior becomes nearly all-pass if the influence of the anti-stable 
part is negligible small.
Also, the error of the GHNA approaches the chosen proper Hankel singular value 
$\varsigma_{5}$, which is exactly the error of the approximation in the 
Hankel semi-norm.

Further examples and tests of the sparse implementation of the GHNA method can
be found in \cite{morWer16}.

\subsection{A Damped Mass-Spring System}

As a second example, a damped mass-spring system with a holonomic constraint
is considered here.
The detailed construction of the system can be found in \cite{morMehS05}.
The vibrations of the resulting system are described by a system of second-order
equations
\begin{align} \label{eq:msd}
  \begin{aligned}
    M\ddot{p}(t) & = Kp(t) + D\dot{p}(t) - G^{T}\lambda(t) + B_{u}u(t),\\
    0 & = Gp(t),\\
    y(t) & = C_{p}p(t),
  \end{aligned}
\end{align}
where $p(t)$ is the position vector, $\lambda(t) \in \mathbb{R}$ is the 
Lagrange multiplier, $K, D \in \mathbb{R}^{g \times g}$ are the tridiagonal
stiffness and damping matrices, $M = \mathrm{diag}(m_{1}, \ldots, m_{g})$ is
the mass matrix and $G = \begin{bmatrix} 1, 0, \ldots, 0, -1 \end{bmatrix}$
is the constraint matrix.
The input matrix is given by $B_{u} = e_{1}$ and three positions of masses are
measured by $C_{p} = \begin{bmatrix} e_{1}, & e_{2}, & e_{g-1}
\end{bmatrix}^{T}$, where $e_{i}$ is the $i$-th column of $I_{g}$.

For the application of the GHNA method, the system \eqref{eq:msd} has to be
rewritten in first-order form.
Therefor, the velocity vector $v(t) = \dot{p}(t)$ is introduced and all states 
are collected in $x(t) = \begin{bmatrix} p(t)^{T}, v(t)^{T}, \lambda(t) 
\end{bmatrix}^{T}$, such that the system \eqref{eq:msd} can be rewritten in the 
form
\begin{align} \label{eq:msdlin}
  \begin{aligned}
    \begin{bmatrix} I_{g} & 0 & 0 \\ 0 & M & 0 \\ 0 & 0 & 0 \end{bmatrix} 
      \dot{x}(t) & = \begin{bmatrix} 0 & I_{g} & 0 \\ K & D & -G^{T} \\ 
      G & 0 & 0\end{bmatrix} x(t) + \begin{bmatrix} 0 \\ B_{u} \\ 0 \end{bmatrix}
      u(t),\\
    y(t) & = \begin{bmatrix} C_{p} & 0 & 0 \end{bmatrix} x(t).
  \end{aligned}
\end{align}
This linearization is an index-$3$ descriptor system.
The number of masses was chosen as $g = 1500$, which leads to $n = 3001$ states
in the linearized system \eqref{eq:msdlin}.
For the computation of the GHNA, the \texttt{ml\_hna\_dss} method from
version 3.0 of the MORLAB toolbox has been used \cite{morBenW17a}.
In this function, the projection-free approach from Section \ref{sec:projfree}
is implemented as mentioned there.
For the computation of the additive decompositions, the right matrix pencil disk
function is used and the generalized Lyapunov equations are solved via the
matrix sign function method; see, for example, \cite{BaiDG97} and
\cite{morBenQ05}.
More details on handling descriptor systems with the MORLAB toolbox can
be found in \cite{morBenW18}.
The computed proper Hankel singular values and the used bound
for the minimal realization of the system can be seen in Figure
\ref{fig:msd_hsvp}.

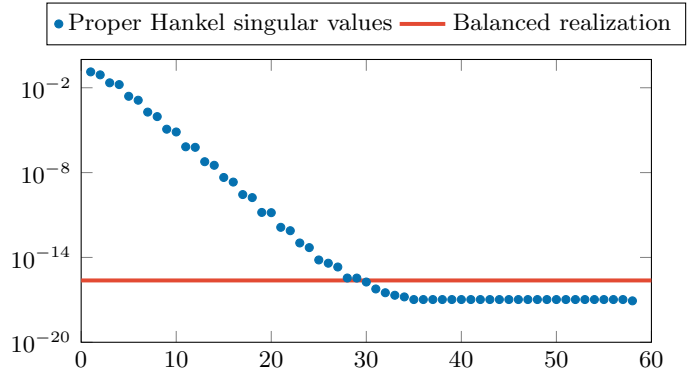
\begin{figure}[t]
  \centering
  \input{graphics/msd_hsvp.tikz}
  \caption{Computed proper Hankel singular values and the tolerance for the 
    balanced realization (mass-spring example).}
  \label{fig:msd_hsvp}
\end{figure}

The computed reduced-order model is of order $6$ ($r = 6$, $\ell_{\infty} = 0$).
So also in this case, the reduced-order model does not contain algebraic
constraints anymore, which means the $\hat{E}$ matrix is regular.
The absolute error of the GHNA is plotted in Figure \ref{fig:msd_r6_error}
with the corresponding $\mathcal{H}_{\infty}$ error bound and the error of the
GBT(SR) reduced-order model for comparison.

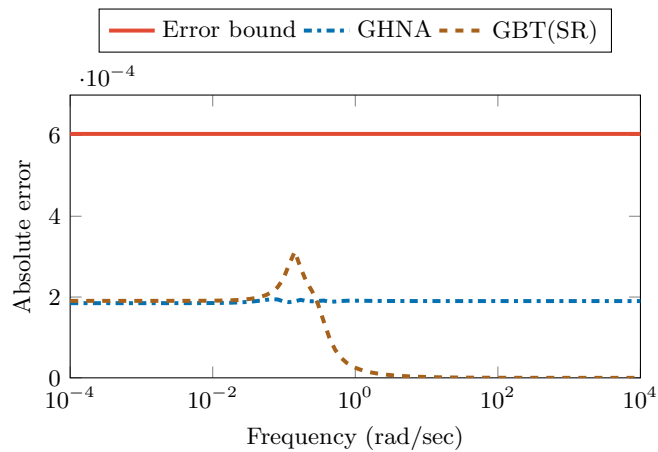
\begin{figure}[t]
  \centering
  \input{graphics/msd_r6_error.tikz}
  \caption{Absolute error of the GHNA and GBT(SR) transfer functions in the 
    spectral norm with the corresponding error bound (mass-spring example).}
  \label{fig:msd_r6_error}
\end{figure}

%% file: graphics/stokes_conv.tikz
\begin{tikzpicture}
  \begin{semilogyaxis}[%
    width  = 0.5\textwidth,
    height = 0.25\textwidth,
    scale only axis = true,
    xmin = 0,
    xmax = 35,
    ymin = 1e-12,
    ymax = 1,
    yminorticks=true,
    ticklabel style = {font = \small},
    yminorticks     = true,
    x label style   = {font = \small, yshift = 0.5em},
    legend columns  = 2,
    legend style    = {
      at={(0.5,1.05)},
      anchor = south,
      nodes  = right,
      font   = \small}
  ]

    \addplot[
      color      = myBlue,
      line width = 1.5pt,
      solid,
      mark         = *,
      mark size    = 1.5pt,
      mark options = {solid, fill = myBlue}
    ]
    coordinates{
      (1,0.691897433940246)
      (2,0.0565730511376398)
      (3,0.0232125663649226)
      (4,0.00184869954523853)
      (5,0.0011355422719378)
      (6,0.000528088393260704)
      (7,0.000173938938417834)
      (8,0.000163480773969617)
      (9,0.000134751404557291)
      (10,1.75440328663032e-05)
      (11,1.73085825060211e-05)
      (12,1.45364094690893e-05)
      (13,1.42605574247256e-05)
      (14,1.37886115060933e-05)
      (15,1.22598598485555e-05)
      (16,1.77452719863269e-06)
      (17,1.76390859264618e-06)
      (18,1.44686364477736e-06)
      (19,1.41551994521706e-06)
      (20,8.64586959313895e-07)
      (21,6.63030101693712e-07)
      (22,3.25416024168523e-08)
      (23,3.20179186830052e-08)
      (24,2.98286015982947e-08)
      (25,1.02720418298425e-08)
      (26,6.93139076155707e-09)
      (27,4.48361577284539e-09)
      (28,4.77450940247874e-11)
      (29,4.43697101623533e-11)
      (30,1.3339197715864e-11)
      (31,1.15991472808279e-11)
      (32,1.73034626174197e-12) 
   };
   \addlegendentry{Normalized residual of $Z_{pc}$};

    \addplot [
      color      = myRed,
      line width = 1.5pt,
      solid,
      mark         = triangle*,
      mark size    = 1.5pt,
      mark options = {solid, fill = myRed}
    ]
    coordinates{
      (1,0.576191415735752)
      (2,0.111058967354886)
      (3,0.0708271124082398)
      (4,0.00712508267741902)
      (5,0.00549382825084091)
      (6,0.00490533417112488)
      (7,0.00274734859337801)
      (8,0.00264034293170628)
      (9,0.00237225790916969)
      (10,0.000640637710074854)
      (11,0.000633627453096434)
      (12,0.000522527028931933)
      (13,0.000514015954611262)
      (14,0.000510605848369983)
      (15,0.000356467776851027)
      (16,0.000355248615873752)
      (17,3.22157554993085e-05)
      (18,3.15983798157106e-05)
      (19,1.66120500055274e-05)
      (20,1.65044440247935e-05)
      (21,1.34139323590155e-05)
      (22,1.33865901662087e-05)
      (23,1.90679448545735e-07)
      (24,1.86990941801148e-07)
      (25,8.36856455160151e-08)
      (26,1.90414551771966e-08)
      (27,1.63246934873726e-08)
      (28,5.28580734362742e-11)
      (29,5.15619273624122e-11)
      (30,1.14248826108031e-11)
      (31,5.82460564390369e-12) 
    };
    \addlegendentry{Normalized residual of $Z_{po}$};

  \end{semilogyaxis}
\end{tikzpicture}

%% file: graphics/stokes_hsvp.tikz
\begin{tikzpicture}
  \begin{semilogyaxis}[%
    width  = 0.5\textwidth,
    height = 0.25\textwidth,
    scale only axis = true,
    xmin  = 0,
    xmax  = 35,
    ymin  = 1e-20,
    ymax  = 10,
    ytick = {1e-20,1e-15,1e-10,1e-5,1e0},
    ticklabel style = {font = \small},
    yminorticks     = true,
    x label style   = {font = \small, yshift = 0.5em},
    legend columns  = 2,
    legend style    = {
      at={(0.5,1.05)},
      anchor = south,
      nodes  = right,
      font   = \small}
  ]

    \addplot[
      color = myBlue,
      mark  = *,
      mark size = 1.5pt,
      only marks
    ] coordinates{
      (1,0.0883420876120351)
      (2,0.00100377679795893)
      (3,0.000311415082555621)
      (4,1.36981651504174e-05)
      (5,1.83695800248467e-06)
      (6,4.74567802100856e-07)
      (7,9.09488297251441e-08)
      (8,1.02848816877585e-08)
      (9,5.80074169859674e-09)
      (10,2.17633690559475e-09)
      (11,5.82082749690712e-11)
      (12,2.94508714174176e-11)
      (13,6.79512018556902e-12)
      (14,2.39118136200952e-12)
      (15,4.56655629398539e-13)
      (16,2.02652348577651e-13)
      (17,7.47896731710239e-14)
      (18,6.88110943293762e-15)
      (19,5.8247096251066e-15)
      (20,6.1588982221573e-16)
      (21,1.9534902215996e-16)
      (22,1.817793911938e-16)
      (23,8.82685441380872e-18)
      (24,8.82685441380872e-18)
      (25,8.82685441380872e-18)
      (26,8.82685441380872e-18)
      (27,8.82685441380872e-18)
      (28,8.82685441380872e-18)
      (29,8.82685441380872e-18)
      (30,8.82685441380872e-18)
      (31,8.7430218861556e-18)
    };
    \addlegendentry{Proper Hankel singular values};

    \addplot[
      color      = myRed,
      line width = 1.5pt,
      solid
    ] coordinates {
      (0,1.93299722058063e-16)
      (35,1.93299722058063e-16)
    };
    \addlegendentry{Balanced realization};
      
  \end{semilogyaxis}
\end{tikzpicture}

%% file: graphics/stokes_r5_error.tikz
\begin{tikzpicture}
  \begin{semilogxaxis}[%
    width  = 0.5\textwidth,
    height = 0.25\textwidth,
    scale only axis = true,
    xmin  = 1e-2,
    xmax  = 1e6,
    ymin  = 1e-11,
    ymax  = 6e-6,
    xtick = {1e-2,1e0,1e2,1e4,1e6},
    xlabel = {Frequency (rad/sec)},
    ylabel = {Absolute error},
    ticklabel style = {font = \small},
    yminorticks     = true,
    x label style   = {font = \small, yshift = 0em},
    y label style   = {font = \small, yshift = 0em},
    legend columns  = 3,
    legend style    = {
      at={(0.5,1.15)},
      anchor = south,
      nodes  = right,
      font   = \small}
  ]
    
    \addplot[
      color      = myRed,
      line width = 1.5pt,
      solid
    ]
    coordinates{
      (0.01,4.84166837585281e-06)
      (1000000,4.84166837585281e-06) 
    };
    \addlegendentry{Error bound};

    \addplot[
      color      = myBlue,
      line width = 1.5pt,
      dash dot
    ]
    coordinates{
      (0.01,1.83397263282454e-06)
      (0.0120450354025878,1.83397266699557e-06)
      (0.0145082877849594,1.8339727155995e-06)
      (0.0174752840000768,1.83397278633206e-06)
      (0.0210490414451202,1.83397289023965e-06)
      (0.0253536449397011,1.83397304004289e-06)
      (0.0305385550883342,1.83397325633143e-06)
      (0.0367837977182863,1.83397357140835e-06)
      (0.0443062145758388,1.83397402744745e-06)
      (0.0533669923120631,1.8339746910328e-06)
      (0.0642807311728432,1.83397565215682e-06)
      (0.0774263682681127,1.83397704646229e-06)
      (0.093260334688322,1.83397906950684e-06)
      (0.112332403297803,1.83398200404296e-06)
      (0.135304777457981,1.83398625819158e-06)
      (0.162975083462064,1.833992426349e-06)
      (0.196304065004027,1.83400136589041e-06)
      (0.236448941264541,1.8340143176053e-06)
      (0.28480358684358,1.83403306937032e-06)
      (0.343046928631492,1.83406019086028e-06)
      (0.413201240011534,1.83409936424143e-06)
      (0.497702356433211,1.83415583282542e-06)
      (0.599484250318941,1.8342369934702e-06)
      (0.722080901838546,1.83435314434187e-06)
      (0.869749002617783,1.83451833885968e-06)
      (1.04761575278967,1.83475114507318e-06)
      (1.26185688306602,1.83507484629813e-06)
      (1.51991108295293,1.83551601884769e-06)
      (1.83073828029537,1.83609948159105e-06)
      (2.20513073990305,1.83683628951816e-06)
      (2.65608778294669,1.83770056292682e-06)
      (3.19926713779738,1.83859320737168e-06)
      (3.85352859371053,1.83930212827287e-06)
      (4.64158883361278,1.83949730136241e-06)
      (5.59081018251222,1.83883958800788e-06)
      (6.73415065775082,1.83727818045806e-06)
      (8.11130830789687,1.83543825357187e-06)
      (9.77009957299226,1.83462898678389e-06)
      (11.76811952435,1.83590227272448e-06)
      (14.174741629268,1.83857978834651e-06)
      (17.0735264747069,1.84014124048208e-06)
      (20.5651230834865,1.83858675040795e-06)
      (24.7707635599171,1.83500711973495e-06)
      (29.8364724028334,1.83294194136143e-06)
      (35.9381366380463,1.83472772283547e-06)
      (43.2876128108306,1.83879659511894e-06)
      (52.1400828799968,1.84101910573283e-06)
      (62.8029144183425,1.83894209490842e-06)
      (75.6463327554629,1.83462597382226e-06)
      (91.116275611549,1.83275934356407e-06)
      (109.749876549306,1.8356271624788e-06)
      (132.194114846603,1.84056048101365e-06)
      (159.228279334109,1.84278845568912e-06)
      (191.791026167249,1.84012094009125e-06)
      (231.012970008316,1.83455179575981e-06)
      (278.255940220713,1.83010760264935e-06)
      (335.160265093884,1.82964694078071e-06)
      (403.701725859656,1.83317474440734e-06)
      (486.260158006535,1.83835421656966e-06)
      (585.702081805666,1.84233281956402e-06)
      (705.480231071865,1.84348072157866e-06)
      (849.753435908644,1.84201477596214e-06)
      (1023.53102189903,1.83935736910107e-06)
      (1232.84673944207,1.83695548313689e-06)
      (1484.96826225447,1.83551450242347e-06)
      (1788.64952905744,1.83500253463683e-06)
      (2154.43469003188,1.83506572160903e-06)
      (2595.02421139974,1.83537249946979e-06)
      (3125.71584968824,1.83572686407172e-06)
      (3764.93580679247,1.83604338734565e-06)
      (4534.87850812858,1.83629675043729e-06)
      (5462.27721768434,1.83648801916958e-06)
      (6579.33224657568,1.83662754495655e-06)
      (7924.82898353917,1.83672723510758e-06)
      (9545.48456661835,1.83679755826346e-06)
      (11497.5699539774,1.83684676987431e-06)
      (13848.8637139387,1.83688103246584e-06)
      (16681.0053720006,1.83690480821366e-06)
      (20092.3300256505,1.83692127085231e-06)
      (24201.2826479438,1.83693265317921e-06)
      (29150.5306282518,1.83694051522241e-06)
      (35111.9173421513,1.83694594209686e-06)
      (42292.428743895,1.83694968635969e-06)
      (50941.3801481637,1.8369522689025e-06)
      (61359.0727341318,1.83695404978764e-06)
      (73907.2203352577,1.83695527768066e-06)
      (89021.5085445039,1.83695612420863e-06)
      (107226.722201032,1.83695670777724e-06)
      (129154.966501488,1.83695711005106e-06)
      (155567.614393047,1.83695738734304e-06)
      (187381.742286038,1.83695757847925e-06)
      (225701.971963392,1.8369577102266e-06)
      (271858.824273294,1.8369578010371e-06)
      (327454.916287773,1.83695786363029e-06)
      (394420.605943766,1.83695790677384e-06)
      (475081.016210279,1.83695793651124e-06)
      (572236.765935022,1.8369579570082e-06)
      (689261.21043497,1.836957971136e-06)
      (830217.568131975,1.83695798087377e-06)
      (1000000,1.83695798758566e-06)
    };
    \addlegendentry{GHNA};

    \addplot[
      color      = myBrown,
      line width = 1.5pt,
      dashed
    ]
    coordinates{
      (0.01,4.43647155836004e-06)
      (0.0120450354025878,4.43647164507448e-06)
      (0.0145082877849594,4.43647176994524e-06)
      (0.0174752840000768,4.43647195128538e-06)
      (0.0210490414451202,4.43647221547937e-06)
      (0.0253536449397011,4.43647259805654e-06)
      (0.0305385550883342,4.43647315192036e-06)
      (0.0367837977182863,4.43647395681682e-06)
      (0.0443062145758388,4.4364751235686e-06)
      (0.0533669923120631,4.4364768182141e-06)
      (0.0642807311728432,4.43647927530768e-06)
      (0.0774263682681127,4.43648284026308e-06)
      (0.093260334688322,4.43648801261671e-06)
      (0.112332403297803,4.43649551700366e-06)
      (0.135304777457981,4.43650640272041e-06)
      (0.162975083462064,4.43652219495802e-06)
      (0.196304065004027,4.43654510337797e-06)
      (0.236448941264541,4.43657833388444e-06)
      (0.28480358684358,4.43662653312994e-06)
      (0.343046928631492,4.43669643346355e-06)
      (0.413201240011534,4.43679778726239e-06)
      (0.497702356433211,4.43694471191142e-06)
      (0.599484250318941,4.4371576152484e-06)
      (0.722080901838546,4.43746595499019e-06)
      (0.869749002617783,4.43791215867699e-06)
      (1.04761575278967,4.43855711960497e-06)
      (1.26185688306602,4.43948781225733e-06)
      (1.51991108295293,4.44082755873548e-06)
      (1.83073828029537,4.44274935158146e-06)
      (2.20513073990305,4.4454919867249e-06)
      (2.65608778294669,4.44937711833909e-06)
      (3.19926713779738,4.454821734126e-06)
      (3.85352859371053,4.46233353398531e-06)
      (4.64158883361278,4.47246488267562e-06)
      (5.59081018251222,4.48568511880155e-06)
      (6.73415065775082,4.50211926518253e-06)
      (8.11130830789687,4.52111748266284e-06)
      (9.77009957299226,4.54070433861831e-06)
      (11.76811952435,4.55713124058888e-06)
      (14.174741629268,4.56493060642508e-06)
      (17.0735264747069,4.55779653263036e-06)
      (20.5651230834865,4.53010586086978e-06)
      (24.7707635599171,4.47827000550797e-06)
      (29.8364724028334,4.40112075558687e-06)
      (35.9381366380463,4.29934251366725e-06)
      (43.2876128108306,4.17466642280618e-06)
      (52.1400828799968,4.02937823299683e-06)
      (62.8029144183425,3.86610526345507e-06)
      (75.6463327554629,3.68769187848538e-06)
      (91.116275611549,3.49723093678099e-06)
      (109.749876549306,3.29825992840392e-06)
      (132.194114846603,3.0946389299374e-06)
      (159.228279334109,2.88960873478068e-06)
      (191.791026167249,2.68443234785683e-06)
      (231.012970008316,2.47782691069666e-06)
      (278.255940220713,2.26698994838481e-06)
      (335.160265093884,2.04980037190612e-06)
      (403.701725859656,1.82694175487197e-06)
      (486.260158006535,1.60280007431909e-06)
      (585.702081805666,1.38469172984973e-06)
      (705.480231071865,1.18076135901245e-06)
      (849.753435908644,9.97493779114021e-07)
      (1023.53102189903,8.3804113303277e-07)
      (1232.84673944207,7.02154407178576e-07)
      (1484.96826225447,5.87443983729859e-07)
      (1788.64952905744,4.90877999898732e-07)
      (2154.43469003188,4.09664688651412e-07)
      (2595.02421139974,3.41467207262646e-07)
      (3125.71584968824,2.84323401116328e-07)
      (3764.93580679247,2.3654710209543e-07)
      (4534.87850812858,1.96677536675027e-07)
      (5462.27721768434,1.63454826300968e-07)
      (6579.33224657568,1.3580098802997e-07)
      (7924.82898353917,1.12800647357367e-07)
      (9545.48456661835,9.36813585030192e-08)
      (11497.5699539774,7.77944027037666e-08)
      (13848.8637139387,6.45968688212377e-08)
      (16681.0053720006,5.36355168464943e-08)
      (20092.3300256505,4.45326176758092e-08)
      (24201.2826479438,3.69737470609739e-08)
      (29150.5306282518,3.06973906400571e-08)
      (35111.9173421513,2.54861632530036e-08)
      (42292.428743895,2.11594325848124e-08)
      (50941.3801481637,1.75671456020916e-08)
      (61359.0727341318,1.45846747537908e-08)
      (73907.2203352577,1.21085229355546e-08)
      (89021.5085445039,1.00527483672512e-08)
      (107226.722201032,8.34599113597333e-09)
      (129154.966501488,6.92900153082041e-09)
      (155567.614393047,5.75258627746439e-09)
      (187381.742286038,4.77590246838545e-09)
      (225701.971963392,3.96504058129966e-09)
      (271858.824273294,3.29184775001021e-09)
      (327454.916287773,2.73295063542131e-09)
      (394420.605943766,2.26894408912257e-09)
      (475081.016210279,1.88371750644913e-09)
      (572236.765935022,1.56389551747747e-09)
      (689261.21043497,1.2983736192993e-09)
      (830217.568131975,1.07793263793184e-09)
      (1000000,8.94918641428547e-10)
    };
    \addlegendentry{GBT(SR)};

  \end{semilogxaxis}
\end{tikzpicture}

%% file: graphics/msd_hsvp.tikz
\begin{tikzpicture}
  \begin{semilogyaxis}[%
    width  = 0.5\textwidth,
    height = 0.25\textwidth,
    scale only axis = true,
    xmin = 0,
    xmax = 60,
    ymin = 1e-20,
    ymax = 1,
    yminorticks = true,
    ticklabel style = {font = \small},
    yminorticks     = true,
    x label style   = {font = \small, yshift = 0.5em},
    legend columns  = 2,
    legend style    = {
      at={(0.5,1.05)},
      anchor = south,
      nodes  = right,
      font   = \small}
  ]
  
    \addplot[
      color = myBlue,
      mark  = *,
      mark size = 1.5pt,
      only marks
    ]
    coordinates{
      (1,0.132660674242346)
      (2,0.0818954168249717)
      (3,0.0223678089134308)
      (4,0.016765444371448)
      (5,0.00249444991948758)
      (6,0.00131402385574456)
      (7,0.000190040410065211)
      (8,9.13012923684003e-05)
      (9,1.17341140022859e-05)
      (10,7.45174441739626e-06)
      (11,6.68948790698851e-07)
      (12,6.1256653030404e-07)
      (13,5.88946161780223e-08)
      (14,3.2872139547875e-08)
      (15,4.52355933242356e-09)
      (16,2.13490316834834e-09)
      (17,2.79563291471503e-10)
      (18,1.70566238722064e-10)
      (19,1.52333015692605e-11)
      (20,1.46482038070353e-11)
      (21,1.34929163102116e-12)
      (22,7.77916250062003e-13)
      (23,1.05275089628259e-13)
      (24,4.95564641930703e-14)
      (25,6.63034318097894e-15)
      (26,3.90193166681088e-15)
      (27,2.10461250397863e-15)
      (28,3.50934063730237e-16)
      (29,3.47196651714326e-16)
      (30,1.82039746861215e-16)
      (31,5.89131739972402e-17)
      (32,3.184347156335e-17)
      (33,2.1121221230151e-17)
      (34,1.61484739494048e-17)
      (35,1.05084479396785e-17)
      (36,1.05084479396785e-17)
      (37,1.05084479396785e-17)
      (38,1.05084479396785e-17)
      (39,1.05084479396785e-17)
      (40,1.05084479396785e-17)
      (41,1.05084479396785e-17)
      (42,1.05084479396785e-17)
      (43,1.05084479396785e-17)
      (44,1.05084479396785e-17)
      (45,1.05084479396785e-17)
      (46,1.05084479396785e-17)
      (47,1.05084479396785e-17)
      (48,1.05084479396785e-17)
      (49,1.05084479396785e-17)
      (50,1.05084479396785e-17)
      (51,1.05084479396785e-17)
      (52,1.05084479396785e-17)
      (53,1.05084479396785e-17)
      (54,1.05084479396785e-17)
      (55,1.05084479396785e-17)
      (56,1.05084479396785e-17)
      (57,1.05084479396785e-17)
      (58,8.514753564097e-18) 
    };
    \addlegendentry{Proper Hankel singular values};

    \addplot[
      color      = myRed,
      line width = 1.5pt,
      solid
    ]
    coordinates{
      (0,2.35850080046537e-16)
      (60,2.35850080046537e-16)
    };
    \addlegendentry{Balanced realization};

  \end{semilogyaxis}
\end{tikzpicture}

%% file: graphics/msd_r6_error.tikz
\begin{tikzpicture}
  \begin{semilogxaxis}[%
    width  = 0.5\textwidth,
    height = 0.25\textwidth,
    scale only axis = true,
    xmin  = 1e-4,
    xmax  = 1e4,
    ymin  = 1e-11,
    ymax  = 7e-4,
    xtick = {1e-4,1e-2,1e0,1e2,1e4},
    xlabel = {Frequency (rad/sec)},
    ylabel = {Absolute error},
    ticklabel style = {font = \small},
    yminorticks     = true,
    x label style   = {font = \small, yshift = 0em},
    y label style   = {font = \small, yshift = 0em},
    legend columns  = 3,
    legend style    = {
      at={(0.5,1.15)},
      anchor = south,
      nodes  = right,
      font   = \small}
  ]

    \addplot[
      color      = myRed,
      line width = 1.5pt,
      solid
    ]
    coordinates{
      (0.0001,0.000603815967398987)
      (10000,0.000603815967398987) 
    };
    \addlegendentry{Error bound};

    \addplot[
      color      = myBlue,
      line width = 1.5pt,
      dash dot
    ]
    coordinates{
      (0.0001,0.000184859687098993)
      (0.000120450354025878,0.000184859703059663)
      (0.000145082877849594,0.000184859726215841)
      (0.000174752840000768,0.000184859759811466)
      (0.000210490414451202,0.000184859808552979)
      (0.000253536449397011,0.00018485987926829)
      (0.000305385550883342,0.000184859981863815)
      (0.000367837977182863,0.00018486013071149)
      (0.000443062145758388,0.000184860346662447)
      (0.000533669923120631,0.000184860659966771)
      (0.000642807311728432,0.000184861114510508)
      (0.000774263682681127,0.000184861773960591)
      (0.00093260334688322,0.000184862730677761)
      (0.00112332403297803,0.000184864118643017)
      (0.00135304777457981,0.000184866132201551)
      (0.00162975083462064,0.000184869053231681)
      (0.00196304065004027,0.000184873290518855)
      (0.00236448941264541,0.000184879436775145)
      (0.0028480358684358,0.000184888351158021)
      (0.00343046928631492,0.000184901278541651)
      (0.00413201240011533,0.000184920021621824)
      (0.00497702356433211,0.000184947188627484)
      (0.00599484250318941,0.000184986548524574)
      (0.00722080901838546,0.000185043537512129)
      (0.00869749002617783,0.000185125975128606)
      (0.0104761575278967,0.000185245063419137)
      (0.0126185688306602,0.000185416751745057)
      (0.0151991108295293,0.000185663535082703)
      (0.0183073828029537,0.000186016670349162)
      (0.0220513073990305,0.000186518534804598)
      (0.0265608778294669,0.000187224150983982)
      (0.0319926713779738,0.000188199128130881)
      (0.0385352859371053,0.000189506864017289)
      (0.0464158883361278,0.000191167027821028)
      (0.0559081018251222,0.000193041611147539)
      (0.0673415065775082,0.000194552911906276)
      (0.0811130830789687,0.000194128735172907)
      (0.0977009957299226,0.000189370355662848)
      (0.1176811952435,0.000187361489891064)
      (0.14174741629268,0.000189337209962559)
      (0.170735264747069,0.000193132170436906)
      (0.205651230834865,0.000190153677265633)
      (0.247707635599171,0.00018953588283882)
      (0.298364724028334,0.000190133616505074)
      (0.359381366380463,0.000191387830179042)
      (0.432876128108306,0.000188190946845951)
      (0.521400828799968,0.0001888400647271)
      (0.628029144183425,0.000190188979130072)
      (0.756463327554629,0.000190870478523342)
      (0.91116275611549,0.00019102043439688)
      (1.09749876549306,0.000190924085262265)
      (1.32194114846603,0.000190751824975688)
      (1.59228279334109,0.000190580226827795)
      (1.91791026167249,0.000190436205508554)
      (2.31012970008316,0.000190324545367835)
      (2.78255940220713,0.000190241652415841)
      (3.35160265093884,0.000190181688423281)
      (4.03701725859656,0.000190139010094728)
      (4.86260158006535,0.000190108952401282)
      (5.85702081805666,0.000190087929878026)
      (7.05480231071865,0.000190073294909628)
      (8.49753435908644,0.000190063138696141)
      (10.2353102189903,0.000190056105670575)
      (12.3284673944207,0.000190051242520851)
      (14.8496826225447,0.000190047883148064)
      (17.8864952905744,0.000190045564150433)
      (21.5443469003188,0.000190043964086595)
      (25.9502421139974,0.000190042860431844)
      (31.2571584968824,0.000190042099348773)
      (37.6493580679247,0.000190041574584785)
      (45.3487850812858,0.000190041212800308)
      (54.6227721768434,0.000190040963395911)
      (65.7933224657568,0.000190040791471945)
      (79.2482898353917,0.000190040672962309)
      (95.4548456661835,0.000190040591273882)
      (114.975699539774,0.000190040534967161)
      (138.488637139387,0.000190040496156145)
      (166.810053720006,0.000190040469404752)
      (200.923300256505,0.000190040450965834)
      (242.012826479438,0.000190040438256499)
      (291.505306282518,0.000190040429496397)
      (351.119173421513,0.000190040423458375)
      (422.92428743895,0.000190040419296589)
      (509.413801481638,0.000190040416428026)
      (613.590727341318,0.000190040414450835)
      (739.072203352578,0.000190040413088032)
      (890.215085445039,0.000190040412148705)
      (1072.26722201032,0.000190040411501262)
      (1291.54966501488,0.000190040411055005)
      (1555.67614393047,0.000190040410747418)
      (1873.81742286038,0.000190040410535409)
      (2257.01971963392,0.00019004041038928)
      (2718.58824273294,0.000190040410288559)
      (3274.54916287773,0.000190040410219136)
      (3944.20605943766,0.000190040410171285)
      (4750.81016210279,0.000190040410138304)
      (5722.36765935022,0.000190040410115571)
      (6892.6121043497,0.000190040410099902)
      (8302.17568131975,0.000190040410089102)
      (10000,0.000190040410081658) 
    };
    \addlegendentry{GHNA};

    \addplot[
      color      = myBrown,
      line width = 1.5pt,
      dashed
    ]
    coordinates{
      (0.0001,0.000190579366869637)
      (0.000120450354025878,0.000190579385825373)
      (0.000145082877849594,0.000190579413326939)
      (0.000174752840000768,0.000190579453226971)
      (0.000210490414451202,0.000190579511115159)
      (0.000253536449397011,0.000190579595101063)
      (0.000305385550883342,0.000190579716950399)
      (0.000367837977182863,0.000190579893733262)
      (0.000443062145758388,0.000190580150215644)
      (0.000533669923120631,0.000190580522329081)
      (0.000642807311728432,0.000190581062205141)
      (0.000774263682681127,0.000190581845479479)
      (0.00093260334688322,0.000190582981890172)
      (0.00112332403297803,0.00019058463065662)
      (0.00135304777457981,0.000190587022795743)
      (0.00162975083462064,0.00019059049350893)
      (0.00196304065004027,0.000190595529191034)
      (0.00236448941264541,0.000190602835675015)
      (0.0028480358684358,0.000190613437335425)
      (0.00343046928631492,0.000190628821062513)
      (0.00413201240011533,0.000190651145550286)
      (0.00497702356433211,0.000190683545792328)
      (0.00599484250318941,0.000190730576656966)
      (0.00722080901838546,0.00019079886024581)
      (0.00869749002617783,0.000190898033146815)
      (0.0104761575278967,0.000191042137775387)
      (0.0126185688306602,0.000191251677149038)
      (0.0151991108295293,0.000191556673123832)
      (0.0183073828029537,0.000192001268651157)
      (0.0220513073990305,0.000192650761997107)
      (0.0265608778294669,0.000193602592865996)
      (0.0319926713779738,0.000195004016374476)
      (0.0385352859371053,0.000197081692324875)
      (0.0464158883361278,0.000200193894390952)
      (0.0559081018251222,0.000204929071351957)
      (0.0673415065775082,0.000212308229354469)
      (0.0811130830789687,0.000224242272105211)
      (0.0977009957299226,0.00024462650647218)
      (0.1176811952435,0.000280610071239107)
      (0.14174741629268,0.000312620571163835)
      (0.170735264747069,0.000272440855771944)
      (0.205651230834865,0.00023833032556012)
      (0.247707635599171,0.000214555842760709)
      (0.298364724028334,0.000183516645579816)
      (0.359381366380463,0.000140352374766643)
      (0.432876128108306,9.91422345872702e-05)
      (0.521400828799968,6.95777764143252e-05)
      (0.628029144183425,5.00934623753811e-05)
      (0.756463327554629,3.71804666533705e-05)
      (0.91116275611549,2.83540636082889e-05)
      (1.09749876549306,2.2104970054842e-05)
      (1.32194114846603,1.75312221499528e-05)
      (1.59228279334109,1.40848623168399e-05)
      (1.91791026167249,1.14243847347559e-05)
      (2.31012970008316,9.33058286214538e-06)
      (2.78255940220713,7.65813606004421e-06)
      (3.35160265093884,6.30737048960752e-06)
      (4.03701725859656,5.20754979948559e-06)
      (4.86260158006535,4.30683283852928e-06)
      (5.85702081805666,3.56612587707936e-06)
      (7.05480231071865,2.95523317560172e-06)
      (8.49753435908644,2.45038046967198e-06)
      (10.2353102189903,2.03257136975058e-06)
      (12.3284673944207,1.6864592195578e-06)
      (14.8496826225447,1.39954594458852e-06)
      (17.8864952905744,1.16159450588924e-06)
      (21.5443469003188,9.64185535669127e-07)
      (25.9502421139974,8.00374683498204e-07)
      (31.2571584968824,6.64422699031293e-07)
      (37.6493580679247,5.51579675679817e-07)
      (45.3487850812858,4.57910695595262e-07)
      (54.6227721768434,3.80153806772382e-07)
      (65.7933224657568,3.15603675307987e-07)
      (79.2482898353917,2.62015885439659e-07)
      (95.4548456661835,2.17527999341887e-07)
      (114.975699539774,1.80594312339111e-07)
      (138.488637139387,1.49931853065947e-07)
      (166.810053720006,1.24475647719046e-07)
      (200.923300256505,1.0334163451245e-07)
      (242.012826479438,8.5795905999743e-08)
      (291.505306282518,7.12291914735177e-08)
      (351.119173421513,5.91356820850804e-08)
      (422.92428743895,4.9095456969153e-08)
      (509.413801481638,4.07598964701167e-08)
      (613.590727341318,3.38395738724282e-08)
      (739.072203352578,2.80942040047196e-08)
      (890.215085445039,2.33242990251423e-08)
      (1072.26722201032,1.93642412687055e-08)
      (1291.54966501488,1.60765324099553e-08)
      (1555.67614393047,1.33470191350364e-08)
      (1873.81742286038,1.10809294783457e-08)
      (2257.01971963392,9.19958215834582e-09)
      (2718.58824273294,7.63765463244993e-09)
      (3274.54916287773,6.34091500088986e-09)
      (3944.20605943766,5.26433900070109e-09)
      (4750.81016210279,4.37054670520533e-09)
      (5722.36765935022,3.62850464682981e-09)
      (6892.6121043497,3.01244829802481e-09)
      (8302.17568131975,2.50098749677754e-09)
      (10000,2.07636375559681e-09) 
    };
    \addlegendentry{GBT(SR)};

  \end{semilogxaxis}
\end{tikzpicture}

%% file: sections/conclusion.tex
\section{Conclusion}
\label{sec:conclusion}
  
An algebraic characterization of descriptor systems with all-pass transfer 
function was proven and based on this explanation, an efficient algorithm for 
the computation of the generalized Hankel-norm approximation was developed
by exploiting the generalized balanced truncation square root method. 
To get a numerically more stable algorithm, an approximate version of the 
Hankel-norm approximation was introduced.
For an efficient practical usage, the introduced method was considered for 
sparse large-scale systems as well as for unstructured dense systems.
The approximation behavior of the method was shown on large- and medium-scale
examples.

In contrast to the approach of Cao, Saltik, and Weiland \cite{morCaoSW15},
the method, introduced in this paper, has several numerical advantages.
It has a more stable and efficient computational behavior, due to the 
fact that the Weierstrass canonical form does not have to be computed.
Also, the introduced method can be applied to more general descriptor systems 
since C-controllability and C-observability were not assumed.

%% file: sections/acknowledgment.tex
This work was supported by the German Research Foundation (DFG) priority program
1897: ``Calm, Smooth and Smart -- Novel Approaches for Influencing
Vibrations by Means of Deliberately Introduced Dissipation'' and 
the German Research Foundation (DFG) research training group 2297 ``MathCoRe'',
Magdeburg.

%% file: sections/appendix.tex
\section*{Appendix}
\label{sec:proofs}

For the proof of Theorem \ref{thm:allpass}, the following lemma is used.

\begin{lemma}
  Each all-pass transfer function $G(s) \in \mathbb{C}^{m \times m}$ is proper.
\end{lemma}

\begin{proof}
  From the definition of all-pass transfer functions it follows that the product
  \begin{align*}
    G(s)G^{T}(-s) & = I_{m}
  \end{align*}
  has to be proper.
  \vspace{1em}

  \noindent\textit{Improper Case:}
  
  First, let's assume that $G$ is an improper transfer function.
  The entries of $G(s)$ are rational polynomials with real coefficients. 
  Since $G$ is improper at least one entry of $G$ must have a higher numerator 
  polynomial degree than the denominator.
  Also, one can observe that for $G^{T}(-s)$ the entries of the matrix are only 
  transposed and coefficients of odd polynomial order change their signs.
  In the single-input single-output (SISO) case $G(s)$ is given by
  \begin{align*}
    G(s) & = \frac{n(s)}{d(s)},
  \end{align*}
  with $\mathrm{deg}(n) > \mathrm{deg}(d)$.
  Let the numerator and denominator of the para-Hermitian function be denoted 
  by $\tilde{n}(s)$ and $\tilde{d}(s)$.
  In this case, it is obvious that for the product it holds
  \begin{align*}
    2 \cdot \mathrm{deg}(n) & = \mathrm{deg}(n\tilde{n}) >
     \mathrm{deg}(d\tilde{d}) = 2 \cdot \mathrm{deg}(d).
  \end{align*}
  So the product is always improper.
  
  In the multi-input multi-output (MIMO) case, it is assumed for simplicity that
  $m = 2$ and that the denominator is equal for all entries and can be factored 
  out such that
  \begin{align*}
    \begin{aligned}
      G(s) & = \frac{1}{d(s)}\begin{bmatrix} n_{11}(s) & n_{12}(s) \\
        n_{21}(s) & n_{22}(s)\end{bmatrix} && \text{and} & G^{T}(-s) & =
        \frac{1}{\tilde{d}(s)}\begin{bmatrix} \tilde{n}_{11}(s) & 
        \tilde{n}_{21}(s) \\ \tilde{n}_{12}(s) & \tilde{n}_{22}(s)\end{bmatrix}.
    \end{aligned}
  \end{align*}
  The resulting product is then
  \begin{align*}
    G(s)G^{T}(-s) & = \frac{1}{d(s)\tilde{d}(s)} \begin{bmatrix} 
      n_{11}(s)\tilde{n}_{11}(s) + n_{12}\tilde{n}_{12} & 
      n_{11}(s)\tilde{n}_{21}(s) + n_{12}\tilde{n}_{22} \\
      n_{21}(s)\tilde{n}_{11}(s) + n_{22}\tilde{n}_{12} & 
      n_{21}(s)\tilde{n}_{21}(s) + n_{22}\tilde{n}_{22} \end{bmatrix}.
  \end{align*}
  If only one of the product entries would have a higher polynomial degree than
  the denominator the argumentation from the SISO case would follow.
  Therefor, we can assume w.l.o.g. that
  \begin{align*}
    \mathrm{deg}(n_{11}) & = \mathrm{deg}(n_{12}) = \mathrm{deg}(d) + 1 = g + 1.
  \end{align*}
  We concentrate on the $(1, 1)$ entry of the matrix product.
  For the resulting polynomial degrees it holds
  \begin{align*}
    \mathrm{deg}(n_{11}\tilde{n}_{11}) & = 2g + 2 \\
    \mathrm{deg}(n_{12}\tilde{n}_{12}) & = 2g + 2 \\
    \mathrm{deg}(n_{11}\tilde{n}_{11} + n_{12}\tilde{n}_{12}) & \leq 2g + 2 \\
    \mathrm{deg}(d\tilde{d}) & = 2g
  \end{align*}
  To get a proper product transfer function, we need that the two highest
  coefficients in $n_{11}\tilde{n_{11}} + n_{12}\tilde{n}_{12}$ cancel out.
  If we now develop the polynomials with
  \begin{align*}
    \begin{aligned}
      n_{11}(s) & = \sum\limits_{k = 0}^{g + 1}{n_{11,k}s^{k}},
        & n_{12}(s) & = \sum\limits_{k = 0}^{g + 1}{n_{12,k}s^{k}},\\
      \tilde{n}_{11}(s) & = \sum\limits_{k = 0}^{g + 1}
        {\tilde{n}_{11,k}s^{k}},
        & \tilde{n}_{12}(s) & = \sum\limits_{k = 0}^{g + 1}
        {\tilde{n}_{12,k}s^{k}}, 
    \end{aligned}
  \end{align*}
  we get that for the first coefficients
  \begin{align*}
    n_{11,g+1}\tilde{n}_{11,g+1} & = -n_{12,g+1}\tilde{n}_{12,g+1}
  \end{align*}
  has to hold, with $\lvert n_{11,g+1} \rvert = 
  \lvert \tilde{n}_{11,g+1} \rvert$ and $\lvert n_{12,g+1} \rvert = 
  \lvert \tilde{n}_{12,g+1} \rvert$.
  Now, if $g + 1$ is even we get
  \begin{align*}
    \begin{aligned}
      n_{11,g+1} & = \tilde{n}_{11,g+1}, &
        n_{12,g+1} & = \tilde{n}_{12,g+1} &&&
        \Rightarrow n_{11,g+1}^{2} & = -\tilde{n}_{12,g+1}^{2},
    \end{aligned}
  \end{align*}
  and if $g + 1$ is odd
  \begin{align*}
    \begin{aligned}
      n_{11,g+1} & = -\tilde{n}_{11,g+1}, &
        n_{12,g+1} & = -\tilde{n}_{12,g+1}, &&&
        \Rightarrow -n_{11,g+1}^{2} & = \tilde{n}_{12,g+1}^{2}.
    \end{aligned}
  \end{align*}
  Both cases are a contradiction to the condition that the coefficients are
  real and non-zero.
  Therefor, an all-pass transfer function cannot be improper.
  \vspace{1em}
    
  \noindent\textit{Strictly Proper Case:}
  Now, let's assume that $G$ is a strictly proper transfer function.
  Using the same argumentation as in the improper case, we get that the product
  of a strictly proper transfer function with its para-Hermitian is also
  strictly proper.
\end{proof}

Now, Theorem \ref{thm:allpass} can be proven.

\begin{proof}
  At first, we can assume w.l.o.g. that $\varsigma = 1$, since the system can be 
  scaled to that case by $\tilde{B} = \varsigma^{-\frac{1}{2}}B$, $\tilde{C} = 
  \varsigma^{-\frac{1}{2}}C$ and $\tilde{D} = \varsigma^{-1}D$.
  \vspace{1em}
    
  \noindent "$\Rightarrow$":\\
  Assume the transfer function $G(s)$ is all-pass. 
  With the previous lemma it follows that $G(s)$ has to be proper. 
  If we consider now the decomposition of the transfer function into its 
  strictly proper and polynomial part $G(s) = G_{sp}(s) + P(s)$, the polynomial 
  one must satisfy
  \begin{align*}
    P(s) & = \sum\limits_{k = 1}^{\infty}M_{k}s^{k},
  \end{align*}
  with $M_{k} = 0$ for all $k \geq 1$.
  In this case, it holds
  \begin{align*}
    \lim\limits_{s \rightarrow \infty}G(s) & = M_{0},
  \end{align*}
  and with the definition of all-pass transfer functions we get
  \begin{align*}
    M_{0}M_{0}^{T} & = G(s)G^{T}(-s) = I_{m}.
  \end{align*}
  So the expressions \eqref{eq:M0} and \eqref{eq:Mk} hold.
  Since the matrix pencil $\lambda E - A$ is assumed to be regular, there are 
  non-singular matrices $Q, Z \in \mathbb{R}^{n \times n}$, which 
  transform the matrix pencil into the following block diagonal structure
  \begin{align*}
    Q(\lambda E - A)Z & = \lambda \begin{bmatrix} E_{f} & 0 \\ 0 & E_{\infty}
      \end{bmatrix} - \begin{bmatrix} A_{f} & 0 \\ 0 & A_{\infty} \end{bmatrix},
  \end{align*}
  where $\lambda E_{f} - A_{f}$ contains all the finite eigenvalues of $\lambda 
  E - A$ and $\lambda E_{\infty} - A_{\infty}$ contains only infinite 
  eigenvalues.
  These transformation matrices can be used on the complete system as a 
  restricted system equivalence transformation
  \begin{align} \label{eq:blkdiag}
    \begin{aligned}
      & (QEZ, QAZ, QB, CZ, D)\\
      & = \left(\begin{bmatrix} E_{f} & 0 \\ 
        0 & E_{\infty} \end{bmatrix}, \begin{bmatrix} A_{f} & 0 \\ 0 & 
        A_{\infty} \end{bmatrix}, \begin{bmatrix} B_{f} \\ B_{\infty} 
        \end{bmatrix}, \begin{bmatrix} C_{f} & C_{\infty} \end{bmatrix},
        D \right).
    \end{aligned}
  \end{align}
  This system decouples into its slow
  \begin{align*}
    E_{f}\dot{x}_{f}(t) & = A_{f}x_{f}(t) + B_{f}u(t),\\
    y_{f}(t) & = C_{f}x_{f}(t), 
  \end{align*}
  and fast subsystem
  \begin{align*}
    E_{\infty}\dot{x}_{\infty}(t) & = A_{\infty}x_{\infty}(t) + 
      B_{\infty}u(t),\\
    y_{\infty}(t) & = C_{\infty}x_{\infty}(t) + Du(t).
  \end{align*}
  The slow subsystem corresponds to the strictly proper part of the transfer 
  function and the fast subsystem to the polynomial part.
  Then, the constant part of the transfer function is then given by
  \begin{align*}
    M_{0} & = D - C_{\infty}A_{\infty}^{-1}B_{\infty},
  \end{align*}
  and the transfer function $G(s)$ has another realization of the form
  \begin{align*}
    G(s) & = C_{f}(sE_{f} - A_{f})^{-1}B_{f} + M_{0},
  \end{align*}
  with invertible matrix $E_{f}$. 
  
  Now, we can use the definition of an all-pass transfer function, i.e.,
  $G(s)G^{T}(-s) = I_{m}$ to get the relation $G^{-1}(s) = G^{T}(-s)$,
  which can be written as
  \begin{align*}
    G^{-1}(s) & = M_{0}^{-1} - M_{0}^{-1}C_{f}(sE_{f} - A_{f} + 
      B_{f}M_{0}^{-1}C_{f})^{-1}B_{f}M_{0}^{-1}\\
    & = G^{T}(-s)\\
    & = M_{0}^{T} + B_{f}^{T}(-sE_{f}^{T} - A_{f}^{T})^{-1}C_{f}^{T}.
  \end{align*}
  The equality $M_{0}^{-1} = M_{0}^{T}$ was already proven above.
  From the $R$-controllability and $R$-observability assumption together with 
  the regularity of $E_{f}$, it follows that there exist invertible
  matrices $T, W \in \mathbb{R}^{n \times n}$ which transform the realization of 
  the inverse transfer function into the realization of the para-Hermitian one, 
  with
  \begin{align} \label{eq:Ef}
    E_{f}^{T}  & = WE_{f}T,\\ \label{eq:Af}
    -A_{f}^{T} & = W(A_{f} - B_{f}M_{0}^{T}C_{f})T,\\ \label{eq:Cf}
    C_{f}^{T}  & = WB_{f}M_{0}^{T},\\ \label{eq:Bf}
    B_{f}^{T}  & = M_{0}^{T}C_{f}T.
  \end{align}
  Now, these expressions can be reformulated.
  From \eqref{eq:Cf} we obtain
  \begin{align*}
    \begin{aligned}
      && C_{f}^{T} & = WB_{f}M_{0}^{T}\\
      \iff && W^{-1}C_{f}^{T} & = B_{f}M_{0}^{T}\\
      \iff && W^{-1}C_{f}^{T}M_{0} & = B_{f}\\
      \iff && B_{f}^{T} & = M_{0}^{T}C_{f}W^{-T}.
    \end{aligned}
  \end{align*}
  From \eqref{eq:Bf} we get
  \begin{align*}
    \begin{aligned}
      && B_{f}^{T} & = M_{0}^{T}C_{f}T\\
      \iff && B_{f}^{T}T^{-1} & = M_{0}^{T}C_{f}\\
      \iff && M_{0}B_{f}^{T}T^{-1}& = C_{f} \\
      \iff && C_{f}^{T} & = T^{-T}B_{f}M_{0}^{T}.
    \end{aligned}
  \end{align*}
  The equation \eqref{eq:Af} can be reformulated as
  \begin{align*}
    \begin{aligned}
      && -A_{f}^{T} & = W(A_{f} - B_{f}M_{0}^{T}C_{f})T\\
    \iff && -W^{-1}A_{f}^{T}T^{-1} & = A_{f} - B_{f}M_{0}^{T}C_{f}\\
    \iff && A_{f} & = -W^{-1}A_{f}^{T}T^{-1} + B_{f}M_{0}^{T}C_{f}\\
    \iff && -A_{f}^{T} & = T^{-T}A_{f}W^{-T} - C_{f}^{T}M_{0}B_{f}^{T}\\
    &&  & = T^{-T}(A_{f} - B_{f} M_{0}^{T}C_{f})W^{-T}.
    \end{aligned}
  \end{align*}
  And as last one, for \eqref{eq:Ef} it holds
  \begin{align*}
    \begin{aligned}
      && E_{f}^{T}  & = WE_{f}T\\
      \iff && W^{-1}E_{f}^{T}T^{-1} & = E_{f}\\
      \iff && E_{f}^{T} & = T^{-T}E_{f}W^{-T}.
    \end{aligned}
  \end{align*}
  Therefor, $T$ and $W^{-T}$ as well as $T^{-1}$ and $W^{T}$ satisfy the same 
  set of equations, which means that $W = T^{-T}$.
  Using this, the expressions \eqref{eq:Ef}--\eqref{eq:Bf} are equivalent to
  \begin{align} \label{eq:Ef2}
    E_{f}^{T}  & = T^{-T}E_{f}T,\\ \label{eq:Af2}
    -A_{f}^{T} & = T^{-T}(A_{f} - B_{f}M_{0}^{T}C_{f})T,\\ \label{eq:Cf2}
    C_{f}^{T}  & = T^{-T}B_{f}M_{0}^{T},\\ \label{eq:Bf2}
    B_{f}^{T}  & = M_{0}^{T}C_{f}T.
  \end{align}
  The expressions \eqref{eq:Ef2}, \eqref{eq:Af2} and \eqref{eq:Bf2} give the $T$ 
  as solution of the following system of matrix equations
  \begin{align*}
    A_{f}T + T^{T}A_{f}^{T} - B_{f}B_{f}^{T} & = 0,\\
    E_{f}T & = T^{T}E_{f}^{T}.
  \end{align*}
  By setting the symmetric matrix
  $\tilde{G}_{pc} = -TE_{f}^{-T} = -E_{f}^{-1}T^{T}$
  the equation system can be rewritten as
  \begin{align} \label{eq:fcontrol}
    A_{f}\tilde{G}_{pc}E_{f}^{T} + E_{f}\tilde{G}_{pc}A_{f}^{T} + B_{f}B_{f}^{T}
      & = 0.
  \end{align}
  Analogously, it follows
  \begin{align} \label{eq:fobserve}
    A_{f}^{T}\tilde{G}_{po}E_{f} + E_{f}^{T}\tilde{G}_{po}A_{f} + C_{f}^{T}C_{f}
      & = 0.
  \end{align}
  with the symmetric matrix $\tilde{G}_{po} = -T^{-1}E_{f}^{-1} = 
  -E_{f}^{-T}T^{-T}$.
  
  For the matrices $\tilde{G}_{pc}$ and $\tilde{G}_{po}$, the following matrix
  product is considered
  \begin{align} \label{eq:fhsv1}
    \begin{aligned}
      \tilde{G}_{pc}E_{f}^{T}\tilde{G}_{po}E_{f} 
        & = (-TE_{f}^{-T})E_{f}^{T}(-T^{-1}E_{f}^{-1})E_{f} = TT^{-1}
        = I_{n_{f}},
    \end{aligned}
  \end{align}
  and also
  \begin{align} \label{eq:fhsv2}
    \begin{aligned}
      \tilde{G}_{po}E_{f}\tilde{G}_{pc}E_{f}^{T}
        & = (-T^{-1}E_{f}^{-1})E_{f}(-TE_{f}^{-T})E_{f}^{T} = TT^{-1}
        = I_{n_{f}}.
    \end{aligned}
  \end{align}
  Additionally, from \eqref{eq:Cf2} it follows
  \begin{align} \label{eq:fconst1}
    \begin{aligned}
      && T^{T}C_{f}^{T} & = B_{f}M_{0}^{T}\\
      \iff && M_{0}B_{f}^{T} - C_{f}T & = 0\\
      \iff && M_{0}B_{f}^{T} + C_{f}G_{pc}E_{f}^{T} & = 0,
    \end{aligned}
  \end{align}
  and from \eqref{eq:Bf2} we obtain
  \begin{align} \label{eq:fconst2}
    M_{0}^{T}C_{f} + B_{f}^{T}\tilde{G}_{po}E_{f} & = 0.
  \end{align}
  
  With the last step all conditions are satisfied on the realization with 
  invertible matrix $E_{f}$. 
  In the next step, the original dimension of the system has to be rebuild by 
  using the block diagonal structure \eqref{eq:blkdiag}.
  This is done by applying appropriate spectral projectors of the deflating 
  subspaces corresponding to the finite eigenvalues of $\lambda E - A$.
  In case of a system in the form \eqref{eq:blkdiag}, the left and right
  spectral projectors are given by
  \begin{align*}
    \begin{aligned}
      \tilde{P}_{\ell} & = \begin{bmatrix} I_{n_{f}} & 0 \\ 0 & 0 \end{bmatrix} 
        && \text{and} & \tilde{P}_{r} & = \begin{bmatrix} I_{n_{f}} & 0 \\ 
        0 & 0 \end{bmatrix}.
    \end{aligned}
  \end{align*}
  Since the matrices $\tilde{G}_{pc}$ and $\tilde{G}_{po}$ are only determined 
  by the system parts corresponding to the finite eigenvalues, they have to be 
  expended accordingly to the spectral projectors by
  \begin{align*}
    \begin{aligned}
      \tilde{G}_{pc} & \rightarrow \begin{bmatrix} \tilde{G}_{pc} & 0 \\ 0 & 0 
        \end{bmatrix} && \text{and} & \tilde{G}_{po} \rightarrow
        & \begin{bmatrix} \tilde{G}_{po} & 0 \\ 0 & 0 \end{bmatrix}.
    \end{aligned}
  \end{align*}
  Using this, the equation \eqref{eq:fcontrol} is equivalent to
  \begin{align} \label{eq:tcontrol}
    \begin{aligned}
      \begin{bmatrix} A_{f} & 0 \\ 0 & A_{\infty} \end{bmatrix}
        \begin{bmatrix} \tilde{G}_{pc} & 0 \\ 0 & 0 \end{bmatrix}
        \begin{bmatrix} E_{f} & 0 \\ 0 & E_{\infty} \end{bmatrix}^{T} + 
        \begin{bmatrix} E_{f} & 0 \\ 0 & E_{\infty} \end{bmatrix}
        \begin{bmatrix} \tilde{G}_{pc} & 0 \\ 0 & 0 \end{bmatrix}
        \begin{bmatrix} A_{f} & 0 \\ 0 & A_{\infty} \end{bmatrix}^{T} \\ 
      +~\tilde{P}_{\ell}\begin{bmatrix} B_{f} \\ B_{\infty} \end{bmatrix}
        \begin{bmatrix} B_{f} \\ B_{\infty} \end{bmatrix}^{T}
        \tilde{P}_{\ell}^{T} = 0
    \end{aligned}
  \end{align}
  and equation \eqref{eq:fobserve} to
  \begin{align} \label{eq:tobserve}
    \begin{aligned}
      \begin{bmatrix} A_{f} & 0 \\ 0 & A_{\infty} \end{bmatrix}^{T}
        \begin{bmatrix} \tilde{G}_{po} & 0 \\ 0 & 0 \end{bmatrix}
        \begin{bmatrix} E_{f} & 0 \\ 0 & E_{\infty} \end{bmatrix} + 
        \begin{bmatrix} E_{f} & 0 \\ 0 & E_{\infty} \end{bmatrix}^{T}
        \begin{bmatrix} \tilde{G}_{po} & 0 \\ 0 & 0 \end{bmatrix}
        \begin{bmatrix} A_{f} & 0 \\ 0 & A_{\infty} \end{bmatrix} \\
      +~\tilde{P}_{r}^{T}\begin{bmatrix} C_{f} & C_{\infty} \end{bmatrix}^{T}
        \begin{bmatrix} C_{f} & C_{\infty} \end{bmatrix}\tilde{P}_{r} = 0.
    \end{aligned}
  \end{align}
  Also, the matrix product in \eqref{eq:fhsv1} becomes
  \begin{align*}
    \begin{bmatrix} \tilde{G}_{pc} & 0 \\ 0 & 0 \end{bmatrix}
      \begin{bmatrix} E_{f} & 0 \\ 0 & E_{\infty} \end{bmatrix}^{T}
      \begin{bmatrix} \tilde{G}_{po} & 0 \\ 0 & 0 \end{bmatrix}
      \begin{bmatrix} E_{f} & 0 \\ 0 & E_{\infty} \end{bmatrix}
        & = \begin{bmatrix} I_{n_{f}} & 0 \\ 0 & 0 \end{bmatrix},
  \end{align*}
  as well as \eqref{eq:fhsv2} with 
  \begin{align*}
    \begin{bmatrix} \tilde{G}_{po} & 0 \\ 0 & 0 \end{bmatrix}
      \begin{bmatrix} E_{f} & 0 \\ 0 & E_{\infty} \end{bmatrix}
      \begin{bmatrix} \tilde{G}_{pc} & 0 \\ 0 & 0 \end{bmatrix}
      \begin{bmatrix} E_{f} & 0 \\ 0 & E_{\infty} \end{bmatrix}^{T}
      & = \begin{bmatrix} I_{n_{f}} & 0 \\ 0 & 0 \end{bmatrix}.
  \end{align*}
  The constraint \eqref{eq:fconst2} becomes
  \begin{align*}
    M_{0}\begin{bmatrix} B_{f} \\ B_{\infty} \end{bmatrix}^{T}
      \tilde{P}_{\ell}^{T} + \begin{bmatrix} C_{f} & C_{\infty} \end{bmatrix}
      \begin{bmatrix} \tilde{G}_{pc} & 0 \\ 0 & 0 \end{bmatrix}
      \begin{bmatrix} E_{f} & 0 \\ 0 & E_{\infty} \end{bmatrix}^{T} 
      & = 0,
  \end{align*}
  and \eqref{eq:fconst1} is equivalent to
  \begin{align*}
    M_{0}^{T}\begin{bmatrix} C_{f} & C_{\infty} \end{bmatrix}\tilde{P}_{r} +
      \begin{bmatrix} B_{f} \\ B_{\infty} \end{bmatrix}^{T}
      \begin{bmatrix} \tilde{G}_{po} & 0 \\ 0 & 0 \end{bmatrix}
      \begin{bmatrix} E_{f} & 0 \\ 0 & E_{\infty} \end{bmatrix} 
      & = 0.
  \end{align*}
  
  As last step in this part, the realization has to be back-transformed into the 
  original one.
  By multiplying \eqref{eq:tcontrol} from the left with $Q^{-1}$ and from
  the right with $Q^{-T}$ we get
  \begin{align*}
    & Q^{-1}\begin{bmatrix} A_{f} & 0 \\ 0 & A_{\infty} \end{bmatrix}
      \begin{bmatrix} \tilde{G}_{pc} & 0 \\ 0 & 0 \end{bmatrix}
      \begin{bmatrix} E_{f} & 0 \\ 0 & E_{\infty} \end{bmatrix}^{T}Q^{-T}\\ 
    & ~~~+{}Q^{-1}\begin{bmatrix} E_{f} & 0 \\ 0 & E_{\infty} \end{bmatrix}
      \begin{bmatrix} \tilde{G}_{pc} & 0 \\ 0 & 0 \end{bmatrix}
      \begin{bmatrix} A_{f} & 0 \\ 0 & A_{\infty} \end{bmatrix}^{T}Q^{-T} \\ 
    & ~~~+{}Q^{-1}\tilde{P}_{\ell}\begin{bmatrix} B_{f} \\ B_{\infty}
      \end{bmatrix} \begin{bmatrix} B_{f} \\ B_{\infty} \end{bmatrix}^{T}
      \tilde{P}_{\ell}^{T}Q^{-T}\\
    & = Q^{-1}\begin{bmatrix} A_{f} & 0 \\ 0 & A_{\infty} \end{bmatrix}Z^{-1}Z
      \begin{bmatrix} \tilde{G}_{pc} & 0 \\ 0 & 0 \end{bmatrix}Z^{T}Z^{-T}
      \begin{bmatrix} E_{f} & 0 \\ 0 & E_{\infty} \end{bmatrix}^{T}Q^{-T} \\
    & ~~~+{}Q^{-1}\begin{bmatrix} E_{f} & 0 \\ 0 & E_{\infty} \end{bmatrix}
      Z^{-1}Z\begin{bmatrix} \tilde{G}_{pc} & 0 \\ 0 & 0 \end{bmatrix}Z^{T}
      Z^{-T}\begin{bmatrix} A_{f} & 0 \\ 0 & A_{\infty} \end{bmatrix}^{T}
      Q^{-T}\\
    & ~~~+{}Q^{-1}\tilde{P}_{\ell}QQ^{-1}\begin{bmatrix} B_{f} \\ B_{\infty} 
      \end{bmatrix}\begin{bmatrix} B_{f} \\ B_{\infty} \end{bmatrix}^{T}Q^{-T}
      Q^{-T}\tilde{P}_{\ell}^{T}Q^{-T}\\
    & = A\mathcal{G}_{pc}E^{T} + E\mathcal{G}_{pc}A^{T} +  
      P_{\ell}BB^{T}P_{\ell}^{T}\\
    & = 0,
  \end{align*}
  with the spectral projection
  \begin{align*}
    P_{\ell} & = Q^{-1}\tilde{P}_{\ell}Q
      = Q^{-1}\begin{bmatrix} I_{n_{f}} & 0 \\ 0 & 0 \end{bmatrix}Q,
  \end{align*}
  and the symmetric matrix
  \begin{align*}
    \mathcal{G}_{pc} & = Z \begin{bmatrix} \tilde{G}_{pc} & 0 \\ 0 & 0 
      \end{bmatrix}Z^{T}.
  \end{align*}
  These are the conditions \eqref{eq:cproj} and \eqref{eq:control}.
  Analogously, the conditions \eqref{eq:oproj} and \eqref{eq:observe} can be
  shown by multiplying \eqref{eq:tobserve} with $Z^{-T}$ from the left and
  $Z^{-1}$ from the right.
  For the condition \eqref{eq:cproj} it holds
  \begin{align*}
    P_{r}\mathcal{G}_{pc}P_{r}^{T} & = Z\begin{bmatrix} I_{n_{f}} & 0 \\ 
      0 & 0 \end{bmatrix}Z^{-1} Z \begin{bmatrix} \tilde{G}_{pc} & 0 \\ 0 & 0 
      \end{bmatrix}Z^{T}Z^{-T}\begin{bmatrix} I_{n_{f}} & 0 \\ 0 & 0 
      \end{bmatrix}Z^{T} \\
    & = Z \begin{bmatrix} \tilde{G}_{pc} & 0 \\ 0 & 0 \end{bmatrix}Z^{T} \\
    & = \mathcal{G}_{pc},
  \end{align*}
  and for \eqref{eq:oproj}
  \begin{align*}
    P_{\ell}^{T} \mathcal{G}_{po} P_{\ell} & = Q^{T}\begin{bmatrix} 
      I_{n_{f}} & 0 \\ 0 & 0 \end{bmatrix}Q^{-T} Q^{T} \begin{bmatrix} 
      \tilde{G}_{po} & 0 \\ 0 & 0 \end{bmatrix}QQ^{-1}\begin{bmatrix} I_{n_{f}} 
      & 0 \\ 0 & 0 \end{bmatrix}Q \\
    & = Q^{T} \begin{bmatrix} \tilde{G}_{po} & 0 \\ 0 & 0 \end{bmatrix}Q \\
    & = \mathcal{G}_{po}.
  \end{align*}
  The condition \eqref{eq:phsv1} for the proper Hankel singular values is then
  \begin{align*}
    \mathcal{G}_{pc}E^{T}\mathcal{G}_{po}E & = 
      Z \begin{bmatrix} \tilde{G}_{pc} & 0 \\ 0 & 0 \end{bmatrix}Z^{T}
      Z^{-T}\begin{bmatrix} E_{f} & 0 \\ 0 & E_{\infty} \end{bmatrix}^{T}
      Q^{-T}\\
    & ~~~\times{} Q^{T} \begin{bmatrix} \tilde{G}_{po} & 0 \\ 0 & 0
      \end{bmatrix}Q Q^{-1}\begin{bmatrix} E_{f} & 0 \\ 0 & E_{\infty}
      \end{bmatrix}Z^{-1} \\
    & = Z \begin{bmatrix} \tilde{G}_{pc} & 0 \\ 0 & 0 \end{bmatrix}
      \begin{bmatrix} E_{f} & 0 \\ 0 & E_{\infty} \end{bmatrix}^{T}
      \begin{bmatrix} \tilde{G}_{po} & 0 \\ 0 & 0 \end{bmatrix}
      \begin{bmatrix} E_{f} & 0 \\ 0 & E_{\infty} \end{bmatrix}Z^{-1} \\
    & = Z\begin{bmatrix} I_{n_{f}} & 0 \\ 0 & 0 \end{bmatrix}Z^{-1} \\
    & = P_{r},
  \end{align*}
  and also the dual condition \eqref{eq:phsv2} can be shown this way.
  For the additional constraint \eqref{eq:addcon1} it holds
  \begin{align*}
    M_{0}^{T}CP_{r} + B^{T}\mathcal{G}_{po}E & = M_{0}^{T}\begin{bmatrix} 
      C_{f} & C_{\infty} \end{bmatrix}Z^{-1}Z\begin{bmatrix} I_{n_{f}} & 0 \\ 
      0 & 0 \end{bmatrix}Z^{-1}\\
    & ~~~+{}\begin{bmatrix} B_{f} \\ B_{\infty} \end{bmatrix}^{T}Q^{-T}
      Q^{T}\begin{bmatrix} \tilde{G}_{po} & 0 \\ 0 & 0 \end{bmatrix}Q 
      Q^{-1}\begin{bmatrix} E_{f} & 0 \\ 0 & E_{\infty} \end{bmatrix}Z^{-1} \\
    & = \left(M_{0}^{T}\begin{bmatrix} C_{f} & C_{\infty} \end{bmatrix}
      \begin{bmatrix} I_{n_{f}} & 0 \\ 0 & 0 \end{bmatrix}\right. \\
    & \left.~~~+{}\begin{bmatrix} B_{f} \\ B_{\infty} \end{bmatrix}^{T}
      \begin{bmatrix} \tilde{G}_{po} & 0 \\ 0 & 0 \end{bmatrix}
      \begin{bmatrix} E_{f} & 0 \\ 0 & E_{\infty} \end{bmatrix}\right)Z^{-1} \\
    & = 0,
  \end{align*}
  and for the dual one \eqref{eq:addcon2} it works the same.
  Hence, all conditions of the characterization are fulfilled.
  \vspace{1em}
  
  \noindent "$\Leftarrow$":\\
  Now, it is assumed that the conditions \eqref{eq:cproj}--\eqref{eq:addcon2}
  hold.
  It has to be shown that the resulting transfer function of the linear
  descriptor system is all-pass.
  Therefor, a reformulation of \eqref{eq:control} is considered
  \begin{align*}
      P_{\ell}BB^{T}P_{\ell}^{T} & = -A\mathcal{G}_{pc}E^{T} - 
        E\mathcal{G}_{pc}A^{T} \\
      & = -A\mathcal{G}_{pc}E^{T} - E\mathcal{G}_{pc}A^{T} + 
        sE\mathcal{G}_{pc}E^{T} - sE\mathcal{G}_{pc}E^{T} \\
      & = (sE - A)\mathcal{G}_{pc}E^{T} + E\mathcal{G}_{pc}(-sE^{T} - A^{T})
  \end{align*}
  The right-hand side of this expression shall be transformed into the form
  of a transfer function and its para-Hermitian.
  It holds
  \begin{align*}
    & (sE - A)^{-1}P_{\ell}BB^{T}P_{\ell}^{T}(-sE^{T} - A^{T})^{-1}\\
    & = \mathcal{G}_{pc}E^{T}(-sE^{T} - A^{T})^{-1} + 
      (sE - A)^{-1}E\mathcal{G}_{pc}\\
    \Rightarrow~ & CP_{r}(sE - A)^{-1}P_{\ell}BB^{T}P_{\ell}^{T}(-sE^{T} - 
      A^{T})^{-1}P_{r}^{T}C^{T}\\
    & = CP_{r}\mathcal{G}_{pc}E^{T}(-sE^{T} - A^{T})^{-1}P_{r}^{T}C^{T} + 
      CP_{r}(sE - A)^{-1}E\mathcal{G}_{pc}P_{r}^{T}C^{T}.
  \end{align*}
  In those parts with the symmetric matrix $\mathcal{G}_{pc}$, there is an 
  additional spectral projector.
  Using the definition of $P_{r}$ and $\mathcal{G}_{pc}$ from the previous
  direction one obtains
  \begin{align*}
    P_{r}\mathcal{G}_{pc} & = Z\begin{bmatrix} I_{n_{f}} & 0 \\ 0 & 0 
      \end{bmatrix}Z^{-1} Z \begin{bmatrix} \tilde{G}_{pc} & 0 \\ 0 & 0 
      \end{bmatrix}Z^{T} \\
    & = Z \begin{bmatrix} \tilde{G}_{pc} & 0 \\ 0 & 0 \end{bmatrix}Z^{T} \\
    & = \mathcal{G}_{pc}.
  \end{align*}
  Hence, it holds
  \begin{align*}
    & CP_{r}(sE - A)^{-1}P_{\ell}BB^{T}P_{\ell}^{T}(-sE^{T} - 
      A^{T})^{-1}P_{r}^{T}C^{T}\\
    & = C\mathcal{G}_{pc}E^{T}(-sE^{T} - A^{T})^{-1}P_{r}^{T}C^{T} + 
      CP_{r}(sE - A)^{-1}E\mathcal{G}_{pc}C^{T}.
  \end{align*}
  Now, the additional constraint \eqref{eq:addcon2} leads to
  \begin{align*}
    & CP_{r}(sE - A)^{-1}P_{\ell}BB^{T}P_{\ell}^{T}(-sE^{T} - 
      A^{T})^{-1}P_{r}^{T}C^{T} \\
    & = -M_{0}B^{T}P_{\ell}^{T}(-sE^{T} - A^{T})^{-1}P_{r}^{T}C^{T} - 
      CP_{r}(sE - A)^{-1}P_{\ell}BM_{0}^{T}
  \end{align*}
  and, inserting the definition of the spectral projectors, we get on the
  left-hand side
  \begin{align*}
      & CP_{r}(sE - A)^{-1}P_{\ell}BB^{T}P_{\ell}^{T}(-sE^{T} - 
        A^{T})^{-1}P_{r}^{T}C^{T} \\
      ={}& CZ\begin{bmatrix} I_{n_{f}} & 0 \\ 0 & 0 \end{bmatrix}Z^{-1}
        (sE - A)^{-1}Q^{-1}\begin{bmatrix} I_{n_{f}} & 0 \\ 0 & 0 \end{bmatrix}
        QB \\
      & \times{}B^{T}Q^{T}\begin{bmatrix} I_{n_{f}} & 0 \\ 0 & 0 \end{bmatrix}
        Q^{-T}(-sE^{T} - A^{T})^{-1}Z^{-T}\begin{bmatrix} I_{n_{f}} & 0 \\ 
        0 & 0 \end{bmatrix}Z^{T}C^{T} \\
      ={}& CZ\begin{bmatrix} I_{n_{f}} & 0 \\ 0 & 0 \end{bmatrix}(sQEZ -
        QAZ)^{-1}\begin{bmatrix} I_{n_{f}} & 0 \\ 0 & 0 \end{bmatrix}QB \\
      & \times{}B^{T}Q^{T}\begin{bmatrix} I_{n_{f}} & 0 \\ 0 & 0 \end{bmatrix}
        (-sZ^{T}E^{T}Q^{T} - Z^{T}A^{T}Q^{T})^{-1}\begin{bmatrix} I_{n_{f}}
        & 0 \\ 0 & 0 \end{bmatrix}Z^{T}C^{T} \\
      = & \begin{bmatrix} C_{f} & C_{\infty} \end{bmatrix}
        \begin{bmatrix} I_{n_{f}} & 0 \\ 0 & 0 \end{bmatrix}
        \left(s\begin{bmatrix} E_{f} & 0 \\ 0 & E_{\infty} \end{bmatrix} - 
        \begin{bmatrix} A_{f} & 0 \\ 0 & A_{\infty} \end{bmatrix}\right)^{-1}
        \begin{bmatrix} I_{n_{f}} & 0 \\ 0 & 0 \end{bmatrix}
        \begin{bmatrix} B_{f} \\ B_{\infty} \end{bmatrix} \\
      & \times{}\begin{bmatrix} B_{f}^{T} & B_{\infty}^{T} \end{bmatrix}
        \begin{bmatrix} I_{n_{f}} & 0 \\ 0 & 0 \end{bmatrix}
        \left(-s\begin{bmatrix} E_{f}^{T} & 0 \\ 0 & E_{\infty}^{T}
        \end{bmatrix}\begin{bmatrix} A_{f}^{T} & 0 \\ 0 & A_{\infty}^{T} 
        \end{bmatrix}\right)^{-1}
        \begin{bmatrix} I_{n_{f}} & 0 \\ 0 & 0 \end{bmatrix}
        \begin{bmatrix} C_{f}^{T} \\ C_{\infty}^{T} \end{bmatrix} \\
      ={}& \left(C_{f}(sE_{f} - A_{f})^{-1}B_{f} + 
        0\cdot(sE_{\infty} - A_{\infty})^{-1}\cdot0\right)\\
      & \times{}\left(B_{f}^{T}(-sE_{f}^{T} - A_{f}^{T})^{-1}C_{f}^{T} +
        0\cdot(-sE_{\infty}^{T} - A_{\infty}^{T})^{-1}\cdot0\right) \\
      ={}& C_{f}(sE_{f} - A_{f})^{-1}B_{f}
        B_{f}^{T}(-sE_{f}^{T} - A_{f}^{T})^{-1}C_{f}^{T}.
  \end{align*}
  For the right hand-side it holds
  \begin{align*}
      & -M_{0}B^{T}P_{\ell}^{T}(-sE^{T} - A^{T})^{-1}P_{r}^{T}C^{T} - 
        CP_{r}(sE - A)^{-1}P_{\ell}BM_{0}^{T}\\
      & = -M_{0}B_{f}^{T}(-sE_{f}^{T} - A_{f}^{T})^{-1}C_{f}^{T} -
        C_{f}(sE_{f} - A_{f})^{-1}B_{f}M_{0}^{T}.
  \end{align*}
  Using the above expressions, the all-pass condition is satisfied
  \begin{align*}
    G(s)G^{T}(-s) & = (C(sE - A)^{-1}B + D)(B^{T}(sE^{T} - 
       A^{T})^{-1}C^{T} + D^{T})\\
    & = (C_{f}(sE_{f} - A_{f})^{-1}B_{f} + M_{0})(B_{f}^{T}(sE_{f}^{T} - 
      A_{f}^{T})^{-1}C_{f}^{T} + M_{0}^{T})\\
    & = C_{f}(sE_{f} - A_{f})^{-1}B_{f}B_{f}^{T}(sE_{f}^{T} - 
      A_{f}^{T})^{-1}C_{f}^{T}\\
    &~~~+{} M_{0}B_{f}^{T}(sE_{f}^{T} - 
      A_{f}^{T})^{-1}C_{f}^{T} + C_{f}(sE_{f} - A_{f})^{-1}B_{f}M_{0}^{T} +
      M_{0}M_{0}^{T}\\
    & = M_{0}M_{0}^{T}\\
    & = I_{m}.
  \end{align*}
  \vspace{-\baselineskip}
\end{proof}
  

%% file: paper.bbl
\begin{thebibliography}{10}

\bibitem{AdaAK71}
V.~M. Adamjan, D.~Z. Arov, and M.~G. Kre{\u\i}n.
\newblock Analytic properties of {S}chmidt pairs for a {H}ankel operator and
  the generalized {S}chur-{T}akagi problem.
\newblock {\em Mathematics of the USSR-Sbornik}, 15(1):31--73, 1971.

\bibitem{morAnt05}
A.~C. Antoulas.
\newblock {\em Approximation of Large-Scale Dynamical Systems}, volume~6 of
  {\em Adv. Des. Control}.
\newblock {SIAM} Publications, Philadelphia, PA, 2005.

\bibitem{BaiDG97}
Z.~Bai, J.~Demmel, and M.~Gu.
\newblock An inverse free parallel spectral divide and conquer algorithm for
  nonsymmetric eigenproblems.
\newblock {\em J. Numer. Math.}, 76(3):279--308, 1997.

\bibitem{morBenQ05}
P.~Benner and E.~S. Quintana-Ort\'{\i}.
\newblock Model reduction based on spectral projection methods.
\newblock In P.~Benner, V.~Mehrmann, and D.~Sorensen, editors, {\em Dimension
  Reduction of Large-Scale Systems}, volume~45 of {\em Lect. Notes Comput. Sci.
  Eng.}, pages 5--45, Berlin/Heidelberg, Germany, 2005. Springer.

\bibitem{morBenQQ05}
P.~Benner, E.~S. Quintana-Ort\'{\i}, and G.~Quintana-Ort\'{\i}.
\newblock Parallel model reduction of large-scale linear descriptor systems via
  balanced truncation.
\newblock In M.~Dayd\'{e}, J.~J. Dongarra, V.~Hern\'{a}ndez, and J.~M.~L.~M.
  Palma, editors, {\em High Performance Computing for Computational Science -
  VECPAR 2004}, volume 3402 of {\em Lecture Notes in Comput. Sci.}, pages
  340--353, Berlin/Heidelberg, Germany, 2005. Springer-Verlag.

\bibitem{morBenS17}
P.~Benner and T.~Stykel.
\newblock Model order reduction for differential-algebraic equations: A survey.
\newblock In Achim Ilchmann and Timo Reis, editors, {\em Surveys in
  Differential-Algebraic Equations IV}, Differential-Algebraic Equations Forum,
  pages 107--160. Springer International Publishing, Cham, March 2017.

\bibitem{morBenW17a}
P.~Benner and S.~W.~R. Werner.
\newblock {MORLAB}-3.0 -- model order reduction laboratory, 2017.
\newblock see also: \url{http://www.mpi-magdeburg.mpg.de/projects/morlab}.

\bibitem{morBenW17}
P.~Benner and S.~W.~R. Werner.
\newblock On the transformation formulas of the {H}ankel-norm approximation.
\newblock {\em Proc. Appl. Math. Mech.}, 17(1):823--824, 2017.

\bibitem{morBenW18}
P.~Benner and S.~W.~R. Werner.
\newblock Model reduction of descriptor systems with the {MORLAB} toolbox.
\newblock {\em {IFAC-PapersOnLine} 9th Vienna International Conference on
  Mathematical Modelling {MATHMOD} 2018, Vienna, Austria, 21--23 February
  2018}, 51(2):547--552, 2018.

\bibitem{morCaoSW15}
X.~Cao, M.~B. Saltik, and S.~Weiland.
\newblock Hankel model reduction for descriptor systems.
\newblock In {\em 2015 54th {IEEE} Conference on Decision and Control (CDC)},
  pages 4668--4673, December 2015.

\bibitem{morGlo84}
K.~Glover.
\newblock All optimal {H}ankel-norm approximations of linear multivariable
  systems and their {L}$^\infty$-error norms.
\newblock {\em Internat. J. Control}, 39(6):1115--1193, 1984.

\bibitem{morGugSW13}
S.~Gugercin, T.~Stykel, and S.~Wyatt.
\newblock Model reduction of descriptor systems by interpolatory projection
  methods.
\newblock {\em {SIAM} J. Sci. Comput.}, 35(5):B1010--B1033, 2013.

\bibitem{KagV92}
B.~K{\aa}gstr{\"o}m and P.~Van~Dooren.
\newblock A generalized state-space approach for the additive decomposition of
  a transfer matrix.
\newblock {\em Numer. Lin. Alg. Appl.}, 1(2):165--181, 1992.

\bibitem{morMehS05}
V.~Mehrmann and T.~Stykel.
\newblock Balanced truncation model reduction for large-scale systems in
  descriptor form.
\newblock In P.~Benner, V.~Mehrmann, and D.~C. Sorensen, editors, {\em
  Dimension Reduction of Large-Scale Systems}, volume~45 of {\em Lect. Notes
  Comput. Sci. Eng.}, pages 83--115. Springer-Verlag, Berlin/Heidelberg,
  Germany, 2005.

\bibitem{SaaKB16-mmess-1.0.1}
J.~Saak, M.~K\"{o}hler, and P.~Benner.
\newblock {M-M.E.S.S.-1.0.1} -- the matrix equations sparse solvers library,
  April 2016.
\newblock see also:~\url{https://www.mpi-magdeburg.mpg.de/projects/mess}.

\bibitem{morSch07}
M.~Schmidt.
\newblock {\em Systematic discretization of input/output maps and other
  contributions to the control of distributed parameter systems}.
\newblock {Ph.D. Thesis}, Technische Universit\"at Berlin, Berlin, 2007.

\bibitem{Sok06}
V.~Sokolov.
\newblock {\em Contributions to the Minimal Realization Problem for Descriptor
  Systems}.
\newblock {D}issertation, Fakult{\"a}t f{\"u}r Mathematik, TU Chemnitz,
  Chemnitz, January 2006.

\bibitem{Sty02}
T.~Stykel.
\newblock {\em Analysis and Numerical Solution of Generalized {L}yapunov
  Equations}.
\newblock {D}issertation, TU Berlin, 2002.

\bibitem{morSty04}
T.~Stykel.
\newblock Gramian-based model reduction for descriptor systems.
\newblock {\em Math. Control Signals Systems}, 16(4):297--319, 2004.

\bibitem{morSty06}
T.~Stykel.
\newblock Balanced truncation model reduction for semidiscretized {S}tokes
  equation.
\newblock {\em Linear Algebra Appl.}, 415(2--3):262--289, 2006.

\bibitem{Sty08}
T.~Stykel.
\newblock Low-rank iterative methods for projected generalized {L}yapunov
  equations.
\newblock {\em Electron. Trans. Numer. Anal.}, 30:187--202, 2008.

\bibitem{morWer16}
S.~Werner.
\newblock Hankel-norm approximation of descriptor systems.
\newblock Master's thesis, Otto-von-Guericke-Universit{\"a}t, Magdeburg,
  Germany, 2016.

\end{thebibliography}
